\documentclass[11pt]{article}
\usepackage[top=1in, right=1.4in, left=1.4in, bottom=1in]{geometry}
\usepackage{inputenc}
\usepackage{amsmath,amsfonts,amssymb,amsthm}
\usepackage{mathtools,mathrsfs}
\usepackage{graphicx}
\usepackage{color}
\usepackage{caption}
\usepackage{enumerate}

\usepackage{tikz}
\usetikzlibrary{calc,decorations.pathreplacing,shapes,matrix,arrows,decorations.markings}

\tikzset{draw half paths/.style 2 args={%
  decoration={show path construction,
    lineto code={
      \draw [#1] (\tikzinputsegmentfirst) -- 
         ($(\tikzinputsegmentfirst)!0.5!(\tikzinputsegmentlast)$);
      \draw [#2] ($(\tikzinputsegmentfirst)!0.5!(\tikzinputsegmentlast)$)
        -- (\tikzinputsegmentlast);
    }
  }, decorate
}}

\tikzset{middlearrow/.style={
        decoration={markings,
            mark= at position 0.5 with {\arrow{#1}} ,
        },
        postaction={decorate}
    }
}

\usepackage{xcolor}
\usepackage{hyperref}
\hypersetup{colorlinks = false, linkbordercolor = {white}}
\usepackage{aliascnt}

\newtheorem{theorem}{Theorem}[section]
\newaliascnt{lemma}{theorem}
\newtheorem{lemma}[lemma]{Lemma} 
\aliascntresetthe{lemma}
\newaliascnt{proposition}{theorem}
\newtheorem{proposition}[proposition]{Proposition}   
\aliascntresetthe{proposition}
\newaliascnt{corollary}{theorem}
\newtheorem{corollary}[corollary]{Corollary}
\aliascntresetthe{corollary}
\newaliascnt{question}{theorem}
\newtheorem{question}[question]{Question}
\aliascntresetthe{question}
\newaliascnt{conjecture}{theorem}

\aliascntresetthe{conjecture}

\theoremstyle{definition}

\newtheorem*{claim1}{Claim 1}
\newtheorem*{claim2}{Claim 2}

\newtheorem*{case1}{Case 1}
\newtheorem*{case2}{Case 2}

\newaliascnt{example}{theorem}
\newtheorem{example}[example]{Example}
\aliascntresetthe{example}

\newaliascnt{definition}{theorem}
\newtheorem{definition}[definition]{Definition}
\aliascntresetthe{definition}

\newtheorem*{acknowledgements}{Acknowledgements}

\theoremstyle{remark}
\newtheorem*{remark}{Remark}
\newtheorem*{remarks}{Remarks}

\newcommand{\mc}[1]{\mathcal{#1}}
\newcommand{\mb}[1]{\mathbb{#1}}
\newcommand{\ms}[1]{\mathscr{#1}}
\newcommand{\G}{\Gamma}
\newcommand{\tr}{(\triangle)}
\newcommand{\tf}{(\therefore)}
\newcommand{\dne}{\hfill $\Box$ \vspace{0.3cm}}
\newcommand{\larr}[1]{\xleftarrow{#1}}
\newcommand{\rarr}[1]{\xrightarrow{#1}}
\newcommand{\ang}[2]{\arg(#1,#2)}
\renewcommand{\emptyset}{\varnothing}

\title{Permutation monoids and MB-homogeneity for graphs and relational structures}
\author{ Thomas D. H. Coleman\footnotemark[1] , David M. Evans\footnotemark[2] , Robert D. Gray\footnotemark[3] $^\text{,}$\footnotemark[4]}

\begin{document}

\maketitle

\footnotetext[1]{School of Mathematics and Statistics, University of St Andrews, St Andrews, KY16 9SS, United Kingdom. Email: \texttt{tdhc@st-andrews.ac.uk}.}
\footnotetext[2]{Department of Mathematics, Imperial College London, South Kensington Campus, London, SW7 2AZ, United Kingdom. Email: \texttt{david.evans@imperial.ac.uk}.}
\footnotetext[3]{School of Mathematics, University of East Anglia, Norwich, NR4 7TJ, United Kingdom. Email: \texttt{robert.d.gray@uea.ac.uk}.}
\footnotetext[4]{This work was supported by the EPSRC
grant EP/N033353/1 `Special inverse
monoids: subgroups, structure, geometry,
rewriting systems and the word problem'.}

\begin{abstract}
In this paper we investigate the connection between infinite permutation monoids and bimorphism monoids of first-order structures. Taking our lead from the study of automorphism groups of structures as infinite permutation groups and the more recent developments in the field of homomorphism-homogeneous structures, we establish a series of results that underline this connection. Of particular interest is the idea of \emph{MB-homogeneity}; a relational structure $\mc{M}$ is MB-homogeneous if every monomorphism between finite substructures of $\mc{M}$ extends to a bimorphism of $\mc{M}$.

The results in question include a characterisation of closed permutation monoids, a Fra\"{i}ss\'{e}-like theorem for MB-homogeneous structures, and the construction of $2^{\aleph_0}$ pairwise non-isomorphic countable MB-homogeneous graphs. We prove that any finite group arises as the automorphism group of some MB-homogeneous graph and use this to construct oligomorphic permutation monoids with any given finite group of units. We also consider MB-homogeneity for various well-known examples of homogeneous structures and in particular give a complete classification of countable homogeneous undirected graphs that are also MB-homogeneous.\\


\end{abstract}

Let $\mc{M}$ be a first-order structure with automorphism group Aut$(\mc{M})$. As ``structure is whatever is preserved by automorphisms" \cite{hodges1993model}, the automorphism group is a key concept in understanding the model theory of $\mc{M}$. Every automorphism of a countably infinite first-order structure $\mc{M}$ is a permutation of the domain $M$ of $\mc{M}$; in this case, we can view Aut$(\mc{M})$ as an infinite permutation group. Much of the existing literature in this field explores connections between infinite permutation group theory and model theory; see \cite{infpermgroups1998} and \cite{oligomorphic1990} for instance. More recent work has studied the endomorphism monoid End$(\mc{M})$ of a first-order structure $\mc{M}$; analogously, these are examples of infinite transformation monoids. By imposing additional conditions on the type of endomorphism of $\mc{M}$, we can obtain various submonoids of End$(\mc{M})$. These various natural monoids of transformations associated with a first-order structure $\mc{M}$ have been studied by Lockett and Truss \cite{lockettgeneric, lockett2014some}.

The principal aim of papers by Cameron and Ne\v{s}et\v{r}il \cite{cameron2006homomorphism}, Ma\v{s}ulovi\'{c} \cite{mavsulovic2007homomorphism}, \cite{masulovic2013}, Ma\v{s}ulovi\'{c} and Pech \cite{masulovic2011oligomorphic}, Lockett and Truss \cite{lockettgeneric, lockett2014some}, and 
Hartman, Hubi\v cka and Ma\v sulovi\'c \cite{Hartman2014} focus on generalising the current theory on infinite permutation groups to the case of infinite transformation monoids, particularly in the case of End$(\mc{M})$. Research on End$(\mc{M})$ is not restricted to finding analogues of results about Aut$(\mc{M})$; understanding End$(\mc{M})$ is key to the study of the \emph{polymorphism clone} Pol$(\mc{M})$ of $\mc{M}$ and hence complexity of constraint satisfaction problems \cite{bodirsky2006constraint}. This connection provides motivation for studying endomorphism monoids of first-order structures, and has been extensively studied by Bodirsky \cite{bodirsky2005core, bodirsky2012complexity}, Bodirsky and Ne\v{s}et\v{r}il \cite{bodirsky2006constraint} and Bodirsky and Pinsker \cite{bodirsky2015schaefer}.

A \emph{bimorphism} $\alpha:\mc{M} \to \mc{N}$ is a bijective homomorphism between two first-order structures $\mc{M}$ and $\mc{N}$. If $\alpha$ is a bijective endomorphism of $\mc{M}$, we say that it is a \emph{bimorphism of $\mc{M}$}. The collection of all bijective endomorphisms of a first-order structure $\mc{M}$ with domain $M$ forms a monoid under the composition operation. We call this the \emph{bimorphism monoid} of the structure $\mc{M}$ and denote it by Bi$(\mc{M})$. Of course, every automorphism of $\mc{M}$ is a bijective homomorphism, but in general the converse does not hold. Hence Bi$(\mc{M})$ contains the automorphism group Aut$(\mc{M})$ (as its group of units), but is also contained in the symmetric group Sym$(M)$ since every element of Bi$(\mc{M})$ is a bijection from $M$ to itself. Therefore, the bimorphism monoid Bi$(\mc{M})$ gives a natural example of a \emph{permutation monoid}; a monoid where each element is a permutation. In fact, as we shall see for any countable set $M$, the closed \emph{submonoids} of the symmetric group Sym$(M)$ under the pointwise convergence topology are precisely the bimorphism monoids of first-order structures $\mc{M}$ on the domain $M$  (\autoref{closed}).

Although it is a natural concept, permutation monoids have not received much attention in the literature. By definition, every permutation monoid is a \emph{group-embeddable monoid}. The study of group-embeddable monoids was a principal interest of early semigroup theorists; a well known result (Ore's Theorem, see \cite{meakin2007groups}) says that a monoid $T$ is embeddable in a group if it is cancellative and satisfies Ore's condition (for $a,b\in T$, then $aT \cap bT\neq \emptyset$). Furthermore, a monoid $T$ is group-embeddable if and only if $M$ has a faithful representation as a monoid of permutations. Equivalently, $M$ is group-embeddable if and only if $M$ is isomorphic to a submonoid of some symmetric group, that is, a permutation monoid. 

A result of Reyes \cite{reyes1970local} (see \cite{oligomorphic1990}) states that a subgroup of the infinite symmetric group is closed if and only if it is the automorphism group of some first-order structure. Examples of highly symmetric infinite permutation groups are abundant in the literature, often arising as automorphism groups of structures with ``nice" structural conditions; specifically, $\aleph_0$-categoricity and homogeneity. A structure $\mc{M}$ is \emph{$\aleph_0$-categorical} if and only if $\mc{M}$ is the unique (up to isomorphism) countable model of the theory Th$(\mc{M})$ of $\mc{M}$, and it is \emph{homogeneous} if every isomorphism between finite substructures of $\mc{M}$ extends to an automorphism of $\mc{M}$. The two are related; every homogeneous structure $\mc{M}$ over a finite relational language is $\aleph_0$-categorical, and an $\aleph_0$-categorical structure $\mc{M}$ is homogeneous if and only if Th$(\mc{M})$ has quantifier elimination \cite{macpherson2011survey}. Finding examples of $\aleph_0$-categorical structures provides corresponding examples of interesting permutation groups. It is a famous theorem of Engeler, Ryll-Nardzewski and Svenonius (see \cite{hodges1993model}) that a structure $\mc{M}$ is $\aleph_0$-categorical if and only if Aut$(\mc{M})$ has finitely many orbits on $M^n$ for every $n\in\mb{N}$; such groups are called \emph{oligomorphic permutation groups}. It follows that a source of oligomorphic permutation groups are automorphism groups of homogeneous structures over a finite relational language. The celebrated theorem of Fra\"{i}ss\'{e} \cite{fraisse1953certaines} gives a characterisation of homogeneous structures, which is used to construct many examples of structures with an oligomorphic automorphism group. Such structures include the countable dense linear order without endpoints $(\mb{Q},<)$, the random graph $R$, and the generic poset $P$. This was followed in time by complete classification results for countable homogeneous structures for posets by Schmerl \cite{schmerl1979countable}, undirected graphs by Lachlan and Woodrow \cite{lachlan1980countable} and directed graphs by Cherlin \cite{cherlin1998classification}.

In the study of infinite permutation monoids and bimorphism monoids of first-order structures, the natural analogue of homogeneity is \emph{MB-homogeneity}. A structure $\mc{M}$ is MB-homogeneous if every finite partial monomorphism of $\mc{M}$ extends to a bimorphism of $\mc{M}$. This notion of homogeneity was first introduced by Lockett and Truss in \cite{lockettgeneric}, in which they determined conditions using MB-homogeneity for the existence of generic bimorphisms of a structure. Further results were shown by the same authors \cite{lockett2014some} where they classified homomorphism-homogeneous (and hence MB-homogeneous) posets in a wide-ranging result. MB-homogeneity is a natural extension of \emph{MM-homogeneity}; where every finite partial monomorphism of $\mc{M}$ extends to a monomorphism of $\mc{M}$. MM-homogeneity has been more widely considered; for instance, Cameron and Ne\v{s}et\v{r}il \cite{cameron2006homomorphism} demonstrated a Fra\"{i}ss\'{e}-theoretic result about construction and uniqueness of MM-homogeneous structures. 

The aim of this paper is to develop the theory of infinite permutation monoids, with a particular focus on bimorphism monoids and MB-homogeneous structures. We begin in \autoref{s2} by recalling the definition of an oligomorphic transformation monoid from \cite{masulovic2011oligomorphic}, and extend results of the same paper by considering other self-map monoids as introduced by Lockett and Truss \cite{lockett2014some}. \autoref{s3} investigates permutation monoids and MB-homogeneity in more detail, including a characterisation of closed permutation monoids (\autoref{closed}) and a Fra\"{i}ss\'{e}-like theorem for MB-homogeneity (Propositions \ref{mapbap}, \ref{mbfraisse}, and \ref{biequiv}). \autoref{s5} is devoted to MB-homogeneous graphs, including establishing some useful properties of MB-homogeneous graphs, introducing the notion of \emph{bimorphism equivalence} (\autoref{defbieq}) and demonstrating that there exist $2^{\aleph_0}$ non-isomorphic MB-homogeneous graphs, each of which is bimorphism equivalent to the random graph $R$ (\autoref{uncmany}). Furthermore, we show that for any finite group $H$ there exists a MB-homogeneous graph $\G$ such that Aut$(\G) \cong H$ (\autoref{mbfrucht}); consequently there exists an oligomorphic permutation monoid with $H$ as its group of units.
Finally, \autoref{s4} assesses a selection of previously known homogeneous structures to determine whether or not they are MB-homogeneous, culminating in a complete classification of countably infinite graphs that are both homogeneous and MB-homogeneous (\autoref{classification}).

The scope of this article is wide, combining ideas from many different areas of mathematics. For background on the theory of semigroups and monoids, we refer the reader to \cite{howie1995fundamentals}. For definitions, notations and conventions in model theory, including the basics of relational structures, see \cite{hodges1993model}. A good introductory text for graph theory is \cite{diestel2000graph}. For more background on infinite permutation groups, see \cite{infpermgroups1998}; their deep connection to automorphism groups of first-order structures is outlined in Chapter 2 of \cite{oligomorphic1990}. Throughout, maps act on the right of their arguments, and we compose maps from left to right. A \emph{relational first-order signature} $\sigma$ consists of a collection of relations $\{\bar{R}_i$ : $i\in I\}$ where each $R_i$ has an arity $n_i\in \mb{N}$ for all $i\in I$. A \emph{$\sigma$-structure} $\mc{M}$ consists of a domain $M$ and subsets $R_i\subseteq M^{n_i}$ interpreting $\bar{R}_i$ in $\sigma$ for each $i\in I$. We follow usual convention for notation regarding relations; for a $\sigma$-structure $\mc{M}$, we write $R_i^{\mc{M}}(\bar{x})$ and say that it holds if $\bar{x}\in R_i\subseteq M^{n_i}$. All structures $\mc{M}$ will be countably infinite unless stated otherwise.

\section{Oligomorphic transformation monoids}\label{s2}

Let $X$ be a countably infinite set and let $T$ be an infinite submonoid of Self$(X)$, the full transformation monoid on $X$. Then $T$ is a transformation monoid with a natural action on the set $X$ via $(x,\alpha)\mapsto x\alpha$ for all $x\in X$ and $\alpha\in T$. We can extend this to an action of $T$ on tuples $\bar{x}\in X^n$ for all $n\in\mb{N}$, with $T$ acting componentwise on $\bar{x}$. We begin by outlining some important definitions regarding the notion of orbits on tuples for a transformation monoid; versions of these appear in \cite{steinberg2010theory}.

\begin{definition} Let $T\subseteq$ Self$(X)$ be a transformation monoid acting on tuples $X^n$ as above, and let $U$ be the group of units of $T$.
\begin{itemize}
\item Define the \emph{forward orbit} of a tuple $\bar{x}$ to be the set \[F(\bar{x}) = \{\bar{y} \in X^n\; : \;(\exists s\in T)(\bar{x}s = \bar{y})\}.\] 
\item Define the \emph{strong orbit} of $\bar{x}$ to be the set \[S(\bar{x}) = \{\bar{y}\in X^n\; : \;(\exists s,t\in T)(\bar{x}s = \bar{y}\text{ and }\bar{y}t = \bar{x})\}\] 
\item Define the \emph{group orbit} of $\bar{x}$ to be the set \[U(\bar{x}) = \{\bar{y}\in X^n\; : \;(\exists s\in U)(\bar{x}s = \bar{y})\}\]
\end{itemize}
\end{definition}

We note immediately that $U(\bar{x})\subseteq S(\bar{x})\subseteq F(\bar{x})$ for any tuple $\bar{x}$ and that if $\bar{y}\in F(\bar{x})$ then $F(\bar{y})\subseteq F(\bar{x})$. Furthermore, the relation $\bar{x}\sim \bar{y}$ if and only if $\bar{x}$ and $\bar{y}$ are in the same strong orbit is an equivalence relation \cite{steinberg2010theory}. In comparison, the forward orbit is reflexive and transitive (and thus a preorder), but it may not be symmetric. We outline a basic lemma regarding these orbits which will be useful throughout the section; the proof of this is omitted.

\begin{lemma}\label{orbits} Let $T$ be a transformation monoid acting on a set of tuples $X^n$. For any tuple $\bar{x}\in X^n$, we have the following:
\begin{enumerate}[(1)]
\item $F(\bar{x}) = \bigcup_{\bar{y}\in F(\bar{x})}S(\bar{y})$;
\item $S(\bar{x}) = \bigcup_{\bar{y}\in S(\bar{x})}U(\bar{y})$. \dne
\end{enumerate}
\end{lemma}

Recall that a permutation group $G\subseteq$ Sym$(X)$ is \emph{oligomorphic} if the action of $G$ componentwise on tuples of $X$ has finitely many orbits on $X^n$ for every $n\in\mb{N}$ \cite{oligomorphic1990}. The next definition, originally of \cite{masulovic2011oligomorphic}, places these concepts in the context of transformation monoids.

\begin{definition}\label{oligodef} 
We say that a transformation monoid $T\subseteq$ Self$(X)$ is \emph{oligomorphic} if the componentwise action of $T$ on tuples of $X$ has finitely many strong orbits on $X^n$ for every $n\in\mb{N}$.
\end{definition}

We note that if $T$ is itself a group, then the strong orbits are the group orbits and the definitions coincide; so any oligomorphic permutation group is an oligomorphic transformation monoid. Our next result provides more connections between oligomorphic permutation groups and oligomorphic transformation monoids, generalising \cite[Lemma 2.10]{masulovic2011oligomorphic}.

\begin{proposition}\label{opgopm} Let $T\subseteq$ Self$(X)$ be a transformation monoid with group of units $U$. If $U$ is an oligomorphic permutation group then $T$ is an oligomorphic transformation monoid.
\end{proposition}

\begin{proof} As $U$ is an oligomorphic permutation group, there are finitely many group orbits $U(\bar{y})$ with $y\in X^n$ for every $n\in\mb{N}$. As every strong orbit $S(\bar{x})$ arises as the union of group orbits $U(\bar{y})$, we conclude that there are at most finitely many strong orbits of $T$ acting on $X^n$ for every natural number $n$ by \autoref{orbits}.
\end{proof}

\begin{remark} By the theorem of Engeler, Ryll-Nardzewski and Svenonius (see \cite{hodges1993model}), $U$ is an oligomorphic permutation group if and only if it is the automorphism group of some $\aleph_0$-categorical structure $\mc{M}$. By this proposition and the fact that Aut$(\mc{M})$ acts as the group of units for any endomorphism monoid $T\in\{$End$(\mc{M}),$ Epi$(\mc{M}),$ Mon$(\mc{M})$, Bi$(\mc{M})$, Emb$(\mc{M})\}$ of $\mc{M}$, we conclude that if $\mc{M}$ is $\aleph_0$-categorical then $T$ is an oligomorphic transformation monoid. (See \cite{lockett2014some} for the definitions of these various transformation monoids.) 
\end{remark}

This result provides us with numerous examples of oligomorphic transformation monoids, with the caveat that they are closely related to $\aleph_0$-categorical structures via their group of units. The main result of this section distances the notion of oligomorphicity in monoids from $\aleph_0$-categoricity by providing a different source of suitable examples; but first we detail some preliminary conditions. In the same way that homogeneous structures over a finite language provide examples of $\aleph_0$-categorical structures (and hence oligomorphic permutation groups), we turn to \emph{homomorphism-homogeneity} to provide examples of oligomorphic transformation monoids. We recall the eighteen different notions of homomorphism-homogeneity as presented in the two papers of Lockett and Truss \cite{lockettgeneric, lockett2014some} in \autoref{xyhomo}.

\begin{table}[h]
\renewcommand{\arraystretch}{1.2}
\centering
\begin{tabular}{c c c c}
\hline 
 & isomorphism (I) & monomorphism (M) & homomorphism (H) \\ 
\hline 
End$(\mc{M})$ (H) & IH & MH & HH \\ 
\hline 
Epi$(\mc{M})$ (E) & IE & ME & HE \\ 
\hline 
Mon$(\mc{M})$ (M) & IM & MM & HM \\ 
\hline 
Bi$(\mc{M})$ (B) & IB & MB & HB \\ 
\hline 
Emb$(\mc{M})$ (I) & II & MI & HI \\ 
\hline 
Aut$(\mc{M})$ (A) & IA & MA & HA \\ 
\hline 
\end{tabular}
\caption{Table of XY-homogeneity: a first-order structure $\mc{M}$ is XY-homogeneous if a finite partial map of type X (column) extends to a map of type Y (row) in the associated monoid.}\label{xyhomo}
\end{table}

We note that if a structure $\mc{M}$ is XY-homogeneous then it is also IY-homogeneous. For shorthand in our next result, denote the monoid of maps of type Y by Y$(\mc{M})$ for some structure $\mc{M}$. For example, H$(\mc{M})$ is the endomorphism monoid of $\mc{M}$, and A$(\mc{M})$ is the automorphism group. A previous observation of Lockett and Truss in \cite[p3]{lockett2014some} says that an endomorphism of a finite relational structure is an automorphism if and only if it is a bijection.

\begin{lemma}\label{finitebiequivalence} If $\mc{A}$ and $\mc{B}$ are finite $\sigma$-structures and $f:\mc{A}\rarr{}\mc{B}$ and $g:\mc{B}\rarr{}\mc{A}$ are bijective homomorphisms, then $\mc{A}\cong\mc{B}$ and $f,g$ are isomorphisms.
\end{lemma}

\begin{proof} The composition map $fg:\mc{A}\rarr{}\mc{A}$ is a bijective endomorphism of $\mc{A}$. By the observation above, $fg$ must be an automorphism of $\mc{A}$. For some $\bar{a}\in A^{n_i}$, if $\neg R_i^{\mc{A}}(\bar{a})$ and $R_i^{\mc{B}}(\bar{a}f)$, then $R_i^{\mc{A}}(\bar{a}fg)$ as $g$ is a homomorphism. Since $fg$ is an automorphism this is a contradiction; so $f$ must preserve non-relations and is therefore an isomorphism. A similar argument applies to show that $g$ is an isomorphism.
\end{proof}

\begin{proposition}\label{xyorbits} Let $\mc{M}$ be an XY-homogeneous structure with domain $M$. Then two tuples $\bar{a}=(a_1,\ldots,a_n)$ and $\bar{b}=(b_1,\ldots,b_n)$ are in the same strong orbit of Y$(\mc{M})$ if and only if there exists a partial isomorphism $f$ of $\mc{M}$ such that $\bar{a}f = \bar{b}$.
\end{proposition}

\begin{proof} Suppose that $\alpha,\beta\in$ Y$(\mc{M})$ are maps such that $\bar{a}\alpha = \bar{b}$ and $\bar{b}\beta = \bar{a}$ respectively; so $a_i\alpha = b_i$ and $b_i\beta = a_i$ for $1\leq i\leq n$. If $\alpha$ sends elements $a_i \neq a_j$ of $\bar{a}$ to elements $a_i\alpha = a_j\alpha$ of $\bar{b}$, then we have that $a_i\alpha\beta = a_j\alpha\beta$ and so $a_i = a_j$ which is a contradiction. Hence the restrictions $\alpha|_{\bar{a}}$ and $\beta|_{\bar{b}}$ are injective maps and so they are also bijections. Now, we consider the maps $\alpha|_{\bar{a}}:\mc{A}\rarr{}\mc{B}$ and $\beta|_{\bar{b}}:\mc{B}\rarr{}\mc{A}$, where $\mc{A},\mc{B}$ are the structures induced by $\mc{M}$ on $\bar{a},\bar{b}$ respectively. By \autoref{finitebiequivalence}, we see that $\mc{A}\cong \mc{B}$ and so we can define $f = \alpha|_{\bar{a}}$. Conversely, we see that if $f$ is an isomorphism between the structures $\mc{A}$ and $\mc{B}$. By XY-homogeneity (and hence IY-homogeneity) of $\mc{M}$, we extend $f$ to a map $\alpha\in $ Y$(\mc{M})$ such that $\bar{a}\alpha = \bar{b}$. Similarly, we extend the map $f^{-1}:\mc{B}\rarr{}\mc{A}$ to a map $\beta\in$ Y$(\mc{M})$ such that $\bar{b}\beta = \bar{a}$ and so $\bar{a}$ and $\bar{b}$ are in the same strong orbit.
\end{proof}

We now move on to proving the final result of this section. A result of \cite[ch2]{oligomorphic1990} states that, for Aut$(\mc{M})$ acting on $M^n$, the number of group orbits on $n$-tuples is finite if the number of group orbits on $k$-tuples of distinct elements is finite for every $k\leq n$. Using the fact from the proof of \autoref{xyorbits} that maps between two tuples in the same strong orbit are bijections, it is not hard to show that a similar result holds for the number of strong orbits of $n$-tuples when Y$(\mc{M})$ is acting on $M^n$. This, together with \autoref{xyorbits}, proves the following theorem.

\begin{theorem}\label{xyotm} If $\mc{M}$ is an XY-homogeneous structure over a finite relational language, then Y$(\mc{M})$ is an oligomorphic transformation monoid.
\end{theorem}

\begin{proof} As $\mc{M}$ is over a finite relational language, it has finitely many isomorphism types on $k$-tuples of distinct elements for any $k\in\mb{N}$. By \autoref{xyorbits}, there are finitely many strong orbits on $k$-tuples of distinct elements for every $k\in \mb{N}$. The result then follows from the observation above.
\end{proof}

Using this theorem, we can find examples of structures with oligomorphic transformation monoids that are not $\aleph_0$-categorical. For instance, any homomorphism-homogeneous poset not in Schmerl's classification (see \cite{mavsulovic2007homomorphism} or \cite{lockett2014some}) has an oligomorphic endomorphism monoid. A notable instance is Corollary 2.2(a) in \cite{cameron2006homomorphism}; there exists a countably infinite graph $\G$ with oligomorphic endomorphism (and monomorphism) monoid but a trivial automorphism group. 

Of particular interest to this paper is the idea of an \emph{oligomorphic permutation monoid}. This follows as a special case of \autoref{oligodef}, where the transformation monoid $T$ is also a permutation monoid. An instance of \autoref{xyotm} states that if $\mc{M}$ is an MB-homogeneous structure over a finite relational language, then Bi$(\mc{M})$ is an oligomorphic permutation monoid. It follows that finding MB-homogeneous structures give interesting examples of permutation monoids.

\section{Permutation monoids and MB-homogeneity}\label{s3}

This section is devoted to the study of infinite permutation monoids and their connection with bimorphisms of structures. This includes a characterisation of closed permutation monoids (\autoref{closed}) and a Fra\"{i}ss\'{e}-like theorem for MB-homogeneous structures (Propositions \ref{mapbap}, \ref{mbfraisse}, and \ref{biequiv}) examples of which provide oligomorphic permutation monoids by \autoref{xyotm}.

\subsection{Permutation monoids and bimorphisms}

To start, recall that the symmetric group Sym$(M)$ on a countably infinite set $M$ has a natural topology given by pointwise convergence, with a basis for open sets given by the cosets of stabilizers of finite tuples. It is well-known (\cite{reyes1970local}, see 4.1.4 in \cite{hodges1993model}) that $G\leq$ Sym$(M)$ is closed under the pointwise convergence topology if and only if it is the automorphism group of some first-order structure $\mc{M}$ on domain $M$. This was generalised by Cameron and Ne\v{s}et\v{r}il \cite{cameron2006homomorphism} to closed submonoids $T\leq$ End$(M)$ under the product topology; this occurs if and only if $T$ is the endomorphism monoid of some first-order structure $\mc{M}$ on domain $M$. Our first result, \autoref{closed}, provides an analogous result for closed permutation monoids; the proof of which is similar to other results along these lines. To do this, we recall the following standard result from point-set topology \cite[Theorem 17.2]{munkres2000topology}: if $Y$ is a subspace of some topological space $X$, then a set $A$ is closed in $Y$ under the subspace topology inherited from $X$ if and only if $A = Y\cap B$, where $B$ is some closed set in $X$. Furthermore, the pointwise convergence topology on Sym$(M)$ is the same as the topology induced on the set by End$(M)$ via the inclusion map: so Sym$(M)$ is a subspace of End$(M)$. For more background on point-set topology, we refer the reader to \cite{munkres2000topology}.

\begin{theorem}\label{closed} Let $M$ be a countable set. A submonoid $T$ of Sym$(M)$ is closed under the pointwise convergence topology if and only if it is the bimorphism monoid of some structure $\mc{M}$ on domain $M$.
\end{theorem}

\begin{proof} We begin with the converse direction. Suppose that $T=$ Bi$(\mc{M})$ is the bimorphism monoid of a structure $\mc{M}$ on domain $M$. As Sym$(M)$ is a subspace of End$(M)$, and Bi$(\mc{M})$ is the intersection of Sym$(M)$ and the closed set End$(\mc{M})$ of End$(M)$, it follows from the result above that Bi$(\mc{M})$ is closed in Sym$(M)$.

For the forward direction, assume that $T$ is a closed submonoid of Sym$(M)$. Define an $n$-ary relation $R_{\bar{x}}$ by \[R_{\bar{x}}(\bar{y}) \Leftrightarrow (\exists s\in T)(\bar{x}s = \bar{y})\] for each $n\in\mb{N}$ and $\bar{x}\in M^n$. Let $\mc{M}$ be the relational structure on $M$ with relations $R_{\bar{x}}$ for all tuples $\bar{x}\in M^n$ and all $n\in\mb{N}$. The proof that $T=$ Bi$(\mc{M})$ is by containment both ways.

As every element of $T$ is already a permutation of the domain $M$ of $\mc{M}$, proving that $T$ acts as endomorphisms on $\mc{M}$ is enough to show that $T\subseteq$ Bi$(\mc{M})$. Assume then that $s\in T$ and $\bar{y}\in M^n$ such that $R_{\bar{x}}(\bar{y})$ holds. Since this happens, there exists an $s'\in T$ such that $\bar{x}s' = \bar{y}$. Therefore $\bar{x}s's = \bar{y}s$ and so $R_{\bar{x}}(\bar{y}s)$ holds. So $T\subseteq$ End$(\mc{M})$ and hence $T\subseteq$ Bi$(\mc{M})$. 

It remains to show that Bi$(\mc{M})\subseteq T$, so suppose that $\alpha\in$ Bi$(\mc{M})$. Our aim is to show that $\alpha$ is a limit point of $T$. Since $T$ is closed, it must contain all its limit points. Note that each $n$-tuple $\bar{x}$ defines a neighbourhood of $\alpha$, consisting of all functions $\beta$ such that $\bar{x}\alpha = \bar{x}\beta$. As $T$ is a monoid, it follows that $R_{\bar{x}}(\bar{x})$ holds and so $R_{\bar{x}}(\bar{x}\alpha)$ holds. By definition of $R_{\bar{x}}$, there exists $s\in T$ such that $\bar{x}\alpha = \bar{x}s$; hence $\alpha$ is a limit point of $T$. Therefore $\alpha\in T$ and so Bi$(\mc{M})\subseteq T$, completing the proof. 
\end{proof}

\begin{remark} It is a well-known result from descriptive set theory that any closed subset $A$ of a Polish space $X$ is itself a Polish space with the induced topology from $X$ (see \cite{kechris2012classical}). As Sym$(M)$ is a Polish space, it follows that Bi$(\mc{M})$ is also a Polish space; so bimorphisms of first-order structures provide natural examples of Polish \emph{monoids}. We leave this area of investigation open.
\end{remark}

Our aim now is to determine a cardinality result for closed submonoids of Sym$(X)$. For any $\bar{x}\in X^n$, denote the \emph{pointwise stabilizer} of $\bar{x}$ to be the set St$(x)$. Note also that as Bi$(\mc{M})$ is a group-embeddable monoid, it is therefore \emph{cancellative}. Recall that a monoid $T$ is cancellative if for all $x,y,z\in T$, then $xy = xz$ implies that $y = z$ and $yx = zx$ also implies that $y = z$.

\begin{proposition} Let $\mc{M}$ be a countably infinite first-order structure. If St$(\bar{x}) \neq \{e\}$ for all tuples $\bar{x}\in M^n$, then $|$Bi$(\mc{M})| = 2^{\aleph_0}$.
\end{proposition}

\begin{proof} Suppose that St$(\bar{x})\neq\{e\}$ for all tuples $\bar{x}\in M^n$. As $\mc{M}$ is countably infinite, we can enumerate elements of $M=\{x_1,x_2,\ldots\}$. Using this enumeration, we define a sequence of tuples $(\bar{x}_k)_{k\in\mb{N}}$ where $\bar{x}_k = (x_1,\ldots,x_k)$ for all $k\in\mb{N}$. Since St$(\bar{x})\neq \{e\}$ for all tuples $\bar{x}$ of $\mc{M}$, for each $k\in\mb{N}$ there exists $t_k\in$ Bi$(\mc{M})$ such that $t_k\neq e$ and $\bar{x}_kt_k = \bar{x}_k$. So as $k\to \infty$ then $(t_k)_{k\in\mb{N}} \to e$ by construction and so $e$ is a limit point of Bi$(\mc{M})$.

Now, for some $\alpha\in$ Bi$(\mc{M})$, consider the sequence $(t_k\alpha)_{k\in\mb{N}}$. Here, $t_k\alpha \neq \alpha$ for any $k\in\mb{N}$; for if $t_k\alpha = \alpha$ for some $k$, then cancellativity of Bi$(\mc{M})$ implies that $t_k = e$, contradicting our earlier assumption. Then $\alpha$ is a limit point for the sequence $(t_k\alpha)_{k\in\mb{N}}$, and so every element of Bi$(\mc{M})$ is a limit point. This means that Bi$(\mc{M})$ is a perfect set and thus has cardinality of the continuum (ch 6, \cite{kechris2012classical}).
\end{proof}

\subsection{Construction of MB-homogeneous structures}

As mentioned in the introduction, Fra\"{i}ss\'{e} demonstrated a powerful tool for constructing homogeneous structures, whose automorphism groups provide interesting examples of infinite permutation groups. Cameron and Ne\v{s}et\v{r}il \cite{cameron2006homomorphism} adapted Fra\"{i}ss\'{e}'s theorem in order to construct MM-homogeneous structures. In this section, we build on these results, providing a Fra\"{i}ss\'{e}-like theorem for constructing MB-homogeneous structures, whose bimorphism monoids provide interesting examples of infinite permutation monoids. Throughout this section $\sigma$ is a relational signature, $\ms{C}$ is a class of finite $\sigma$-structures and, following convention, we write $A,B$ to mean the $\sigma$-structures on domains $A,B$. We refer the reader to \cite{hodges1993model} for further background on model theory.

There are two main properties that a class of finite structures $\ms{C}$ needs to have that guarantee the existence of a countable homogeneous structure with age $\ms{C}$. The first is the \emph{joint embedding property} (JEP); this property ensures that we can construct a countable structure $\mc{M}$ with age $\ms{C}$. The second is the \emph{amalgamation property} (AP); this ensures our constructed structure $\mc{M}$ has the \emph{extension property}, which in turn implies homogeneity of $\mc{M}$. When building a countable MB-homogeneous structure $\mc{M}$ with age $\ms{C}$, we still need the JEP to ensure that $\mc{M}$ has the age we want but we need a different amalgamation property to ensure MB-homogeneity rather than standard homogeneity. As we are aiming to extend to bijective endomorphisms, we require a ``back and forth" style argument to ensure that the extended map is indeed bijective. Due to the fact that bimorphisms are not automorphisms in general, we require \emph{two} amalgamation conditions, and hence two extension conditions, to ensure that $\mc{M}$ is MB-homogeneous. We begin this section by examining the required extension properties.

If $\mc{M}$ is MB-homogeneous then $\mc{M}$ is MM-homogeneous; it follows (by Proposition 4.1(b) of \cite{cameron2006homomorphism}) that $\mc{M}$ must have the \emph{mono-extension property} (MEP), which takes care of the forward extension:

\begin{quotation} (MEP) For all $A,B\in$ Age$(\mc{M})$ with $A\subseteq B$ and monomorphism $f:~A\rarr{}~\mc{M}$, there exists a monomorphism $g:B\rarr{}\mc{M}$ extending $f$.
\end{quotation}

Any suitable ``back" condition should express in its statement the key difference between a monomorphism of $\mc{M}$ and a bimorphism. This difference is the existence of a preimage; for any $\alpha\in$ Bi$(\mc{M})$ and substructure $A\subseteq \mc{M}$ there exists a substructure $B\subseteq \mc{M}$ such that $B\alpha = A$. Note that this is not true for an arbitrary monomorphism $\beta$ of $\mc{M}$ as $\mc{M}\beta$ is not required to equal $\mc{M}$. The existence of a preimage for every extended map is the condition we wish to ensure in our ``back" extension property, and this motivates our next definition.

\begin{definition} Let $A,B$ be two $\sigma$-structures. We say that an injective map $\bar{f}:A\rarr{}B$ is an \emph{antimonomorphism} if and only if $\neg R^A(a_1,\ldots,a_n)$ implies $\neg R^B(a_1\bar{f},\ldots,a_n\bar{f})$ for all $n$-ary relations $R$ of $\sigma$.
\end{definition}

\begin{remarks} Note that any map that is both a monomorphism and an antimonomorphism is an embedding. 

It is an easy exercise to show that the function composition of two antimonomorphisms is again an antimonomorphism. In particular, it is important to note that the composition of an antimonomorphism and an embedding is an antimonomorphism.
\end{remarks}

We now state and prove a lemma stating that a preimage of a monomorphism is an antimonomorphism. For any bijective function $f:A\to B$ we use $f^{-1}:B\to A$ to denote the inverse of $f$. 

\begin{lemma}\label{antimono} Let $A,B$ be two $\sigma$-structures, and suppose that $f:A\rarr{}B$ is a bijection. Then $f$ is a monomorphism if and only if $f^{-1}$ is an antimonomorphism.
\end{lemma} 

\begin{proof} 
Assume that $f:A\to B$ is a bijection that is not a monomorphism. This happens if and only if there is some $n$-ary relation $R_i$ and some $n$-tuple $\bar{a}$ of $A$ where $R_i^A(\bar{a})$ holds but $R_i^B(\bar{a}f)$ does not hold. As $f$ is bijective, $f^{-1}$ exists and $\bar{a}ff^{-1} = \bar{a}$. Since $R_i^A(\bar{a})$ holds, this means that the non-relation $R_i^B(\bar{a}f)$ is not preserved by $f^{-1}$ and so $f^{-1}$ is not an antimonomorphism.
\end{proof}

\begin{remarks} 
Following this lemma, if $f:A\to B$ is a bijective homomorphism, we write $f^{-1} = \bar{f}: B\to A$ in order to emphasise that the inverse of $f$ is an antimonomorphism. Similarly, if $\bar{f}:B\to A$ is a bijective antimonomorphism, we write $\bar{f}^{-1} = f: A\to B$ to emphasise that the inverse of $\bar{f}$ is a monomorphism. The context for when we use this notation should be clear.

By restricting the codomain of a monomorphism $f:A\rarr{}B$ to the image, we see that $f':A\rarr{}Af$ is a bijective homomorphism and therefore $\bar{f}':Af\rarr{}A$ is an antimonomorphism by the above lemma. Similarly, by restricting the codomain of an antimonomorphism $\bar{f}:B\rarr{}A$ to the image, we see that $\bar{f}':B\rarr{}B\bar{f}$ is a bijective antimonomorphism and so we obtain a bijective homomorphism $f':B\bar{f}\rarr{}B$.

It is an immediate corollary of this result that if $A,B,C$ are $\sigma$-structures, and $f:A\to B$, $g:B\to C$ are bijective homomorphisms, then $(fg)^{-1} = g^{-1}f^{-1}: C\to A$ is a bijective antimonomorphism.  
\end{remarks}

We use antimonomorphisms to express the backward extension in the \emph{antimono-extension property} (AMEP):

\begin{quotation} (AMEP) For all $A,B\in$ Age$(\mc{M})$ with $A\subseteq B$ and antimonomorphism $\bar{f}:A\rarr{}\mc{M}$, there exists a antimonomorphism $\bar{g}:B\rarr{}\mc{M}$ extending $\bar{f}$.
\end{quotation}

These properties prove to be necessary and sufficient for MB-homogeneity. It is easy to show that for any finite substructure $B\subseteq \mc{M}$ of a countable structure $\mc{M}$ and $\alpha\in$ Bi$(\mc{M})$, there exists a finite substructure $C\subset \mc{M}$ such that $C\alpha = B$.

\begin{proposition}\label{mepbep} Let $\mc{M}$ be a countable structure. Then $\mc{M}$ is MB-homogeneous if and only if $\mc{M}$ has the AMEP and the MEP.
\end{proposition}

\begin{proof}

Suppose that $\mc{M}$ is MB-homogeneous. Then $\mc{M}$ is also MM-homogeneous and has the MEP from Proposition 4.1(a) of \cite{cameron2006homomorphism}. 
Let $A\subseteq B\in$ Age$(\mc{M})$ and $\bar{f}:A\to\mc{M}$ be an antimonomorphism. 
Without loss of generality, we can assume that $A\subseteq B\subset\mc{M}$.
Restrict the codomain of $\bar{f}$ to its image to find a bijective antimonomorphism $\bar{f}':A\to A\bar{f}$ between finite substructures of $\mc{M}$.
By \autoref{antimono}, $f': A\bar{f}\to A$ is a monomorphism; as $\mc{M}$ is MB-homogeneous, extend $f'$ to a bimorphism $\alpha$ of $\mc{M}$. 
From the remark above, there exists a $C\subset \mc{M}$ such that $\alpha|_{C} = g:C\to B$ is a bijective homomorphism. By \autoref{antimono}, define $\bar{g}:B\to C$ to be an antimonomorphism. 
It remains to show that $\bar{g}$ extends $\bar{f}$. Indeed, for $a\in A$ we have that \[a\bar{f} = a\bar{f}\alpha\bar{\alpha} = a\bar{f}f'\bar{\alpha} = a\bar{\alpha} = a\bar{g}\] as $\alpha$ extends $f'$ and $\bar{\alpha}$ extends $g$. Therefore $\mc{M}$ has the AMEP.

Conversely, suppose that $\mc{M}$ has the AMEP and the MEP, and let $f:A\rarr{}B$ be a monomorphism between finite substructures of $\mc{M}$. We shall extend $f:A\to B$ to a bimorphism $\alpha$ of $\mc{M}$ in stages using a back and forth argument. Set $A_0 = A, B_0 = B$ and $f_0 = f$. At a typical stage, we have a bijective monomorphism $f_k:A_k\rarr{}B_k$ extending $f$, where $A_i\subseteq A_{i+1}$ and $B_i\subseteq B_{i+1}$ for all $i\leq k$. As $\mc{M}$ is countable, we can enumerate the points $M = \{m_0,m_1,\ldots\}$. When $k$ is even, pick a point $m_i$, where $i$ is the smallest number such that $m_i\notin$ dom $f$. So, $A_k\cup \{m_i\}$ is a substructure of $\mc{M}$ containing $A_k$. Using the MEP, extend $f_k$ to a monomorphism $f'_{k+1}:A_k\cup\{m_i\}\rarr{}\mc{M}$. By restricting the codomain of $f_{k+1}'$ we have that $f_{k+1}:A_k\cup\{m_i\}\rarr{}B_k\cup\{m_if'_{k+1}\}$ is a bijective homomorphism extending $f_k$. If $k$ is odd, note that $\bar{f}_k:B_k\rarr{}A_k$ is a bijective antimonomorphism by \autoref{antimono}. Select a point $m_i$, where $i$ is the least number such that $m_i\notin$ dom $\bar{f}_k$; so $B_k\cup\{m_i\}$ is a substructure of $\mc{M}$. Using the AMEP, extend $\bar{f}_k$ to an antimonomorphism $\bar{f}'_{k+1}:B_k\cup\{m_i\}\rarr{}\mc{M}$. Again, by restricting the codomain we obtain a bijective antimonomorphism $\bar{f}_{k+1}:B_k\cup\{m_i\}\rarr{}A_k\cup\{m_i\bar{f}'_{k+1}\}$ extending $\bar{f}_k$; so $f_{k+1}:A_k\cup\{m_i\bar{f}'_{k+1}\}\rarr{}B_k\cup\{m_i\}$ is a bijective homomorphism extending $f_k$. By ensuring that every $m_i$ appears at both an odd and even stage, countably many applications of this procedure yield a bimorphism $\alpha$ of $\mc{M}$ extending $f$.
\end{proof}

\begin{remark} We note here that this process of extending one point at a time eventually creates a bimorphism. In a similar fashion to Dolinka \cite[Section 3]{dolinka2014bergman} we can present equivalent conditions to \autoref{mepbep} using one point extensions. 

\begin{quote}(1PAMEP/1PMEP) Suppose that $A\subseteq B\in$ Age$(\mc{M})$ and there is a (anti)monomorphism $f:A\rarr{}\mc{M}$. If $|B\smallsetminus A| = 1$ then there exists a (anti)monomorphism $g:B\rarr{}M$ extending $f$.
\end{quote}
A straightforward proof by induction shows that a structure $\mc{M}$ has both these conditions, if and only if it also has the AMEP and MEP. These easier-to-handle properties are useful in determining whether or not a structure is MB-homogeneous.
\end{remark}

We now turn our attention to the actual process of construction. As we also need to ensure MM-homogeneity, it would make sense to take our ``forth" amalgamation condition to be the \emph{mono-amalgamation property} introduced in \cite{cameron2006homomorphism}.

\begin{quotation}(MAP) Let $\ms{C}$ be a class of finite structures. For any $A,B_1,B_2\in\ms{C}$ and any maps $f_i:A\rarr{}B_i$ (for $i = 1,2$) such that $f_1$ is a monomorphism and $f_2$ is an embedding, there exists $C\in\ms{C}$ and monomorphisms $g_i:B_i\rarr{}C$ (for $i=1,2$) such that $f_1g_1 = f_2g_2$ and $g_1$ is an embedding.
\end{quotation}

Similar to the AMEP, we use antimonomorphisms to set out the ``back" amalgamation condition, which we call the \emph{antimono-amalgamation property} (AMAP).

\begin{quotation}(AMAP, see \autoref{dbxyap}) Let $\ms{C}$ be a class of finite structures. For any elements $A,B_1,B_2\in\ms{C}$, antimonomorphism $\bar{f}_1:A\rarr{}B_1$ and embedding $f_2:A\rarr{}B_2$, there exists $D\in\ms{C}$, embedding $g_1:B_1\rarr{}D$ and antimonomorphism $\bar{g}_2:B_2\rarr{}D$ such that $\bar{f}_1g_1 = f_2\bar{g}_2$.
\end{quotation}

\begin{figure}[h]
\centering
\begin{tikzpicture}[node distance=2cm]
\node(D) {$\exists D$};
\node(R) [below right of=D] {$B_2$};
\node(L) [below left of=D] {$B_1$};
\node(H) [below right of=L] {$A$};
\path[->,thick,dotted] (R) edge node[above right] {$\bar{g}_2$} (D);
\path[right hook->,thick,dotted] (L) edge node[above left] {$g_1$} (D);
\path[right hook->] (H) edge node[below right] {$f_2$} (R);
\path[->] (H) edge node[below left] {$\bar{f}_1$} (L);
\end{tikzpicture}
\caption{The antimono-amalgamation property (AMAP)}\label{dbxyap}
\end{figure}
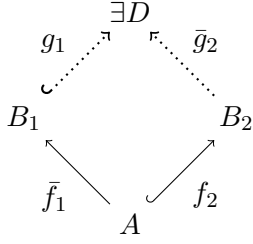

We have enough now to prove the first of our propositions that detail our Fra\"{i}ss\'{e}-like construction.

\begin{proposition}\label{mapbap} Let $\mc{M}$ be an MB-homogeneous structure. Then Age$(\mc{M})$ has the MAP and the AMAP.
\end{proposition}

\begin{proof} As $\mc{M}$ is MB-homogeneous it is necessarily MM-homogeneous; by Proposition 4.1(b) of \cite{cameron2006homomorphism}, Age$(\mc{M})$ has the MAP.

Suppose then that $A,B_1,B_2$ are structures in Age$(\mc{M})$, and assume we have an antimonomorphism $\bar{f}_1:A\rarr{}B_1$ and an embedding $f_2:A\rarr{}B_2$. Without loss of generality, further assume that $A,B_1,B_2\subseteq \mc{M}$ and that $f_2$ is the inclusion map. From \autoref{antimono}, $\bar{f}_1$ induces a bijective monomorphism $f_1:A\bar{f}_1\rarr{}A$. As $\mc{M}$ is MB-homogeneous we extend $f_1$ to some $\alpha\in$ Bi$(\mc{M})$. Restricting $\alpha$ to $B_1$ yields a bijective monomorphism $\alpha|_{B_1}:B_1\rarr{}B_1\alpha$, where $B_1\alpha \supseteq A$. Let $D$ be the structure induced by $\mc{M}$ on $B_1\alpha\cup B_2$. It follows that $A\subseteq D$. As $\alpha$ is surjective, we define $C$ to be the structure such that $C\alpha = D$. Set $g_1:B_1\rarr{}C$ to be the inclusion map. As $\alpha:\mc{M}\rarr{}\mc{M}$ is a bijective homomorphism, then $\bar{\alpha}:\mc{M}\rarr{}\mc{M}$ is an antimonomorphism by \autoref{antimono}. Hence, $\bar{\alpha}|_{B_2}:B_2\rarr{}B_2\bar{\alpha}\subseteq C$ and define $\bar{g}_2:B_2\rarr{} C$ to be this antimonomorphism. It is easy to check that $\bar{f}_1g_1 = f_2\bar{g}_2$ and so Age$(\mc{M})$ has the AMAP.
\end{proof}

For our next result, recall that a class of finite structures $\ms{C}$ has the JEP if for all $A,B\in\ms{C}$ there exists a $D\in\ms{C}$ such that $A$ and $B$ both have embeddings into $D$. 
For a relational structure $\mc{M}$, the \emph{age} Age$(\mc{M})$ of $\mc{M}$ is the class of all finite structures that can be embedded into $\mc{M}$.

\begin{proposition}\label{mbfraisse} If $\ms{C}$ is a class of finite relational structures that is closed under isomorphisms and substructures, has countably many isomorphism types and has the JEP, MAP and AMAP, then there exists an MB-homogeneous structure $\mc{M}$ such that Age$(\mc{M}) = \ms{C}$. 
\end{proposition}
% \subsubsection*{New/alternative proof of 2.7}
\begin{proof}
We build $\mc{M}$ as a countable union of finite structures $M_i$ from $\ms{C}$ with $M_k \subseteq M_{k+1}$ for $k \in \mathbb{N}$. 
The sets $M_k$ $(k \in \mathbb{N})$ are defined inductively in the following way. 
Begin by setting $M_0$ to be an arbitrary fixed structure $A \in \ms{C}$.
As the number of isomorphism types in $\ms{C}$ is countable, we can choose a countable set $S$ of pairs $(A,B)$ where $A\subseteq B\in\ms{C}$ such that every pair $A'\subseteq B'\in\ms{C}$ is represented by a pair $(A,B)$ in $S$.  
This means that for every pair $A'\subseteq B'\in\ms{C}$ there is a pair $(A,B) \in S$ and 
an isomorphism from $B$ to $B'$ which restricts to give an isomorphism between $A$ and $A'$. 
Let $[\mathbf{m}] = \{n\in\mb{N}\; :\; n\equiv m\mod{3}\}$, where $m = 1,2$. Define bijections $\beta_m:[\mathbf{m}]\times \mb{N}\to [\mathbf{m}]$ such that $\beta_m(i,j) \geq i$ for $m = 1,2$. 

Assume first that $k\equiv 0\mod{3}$. 
Let $\mathcal{T} = \{T_0,T_3,T_6,\ldots\}$ be a countable subset of $\ms{C}$ such that each member of $\ms{C}$ is isomorphic to a structure from $\mathcal{T}$. Such a set exists since $\ms{C}$ is assumed to have countably many isomorphism types. By JEP there is a structure $D$ in $\ms{C}$ into which both $M_k$ and $T_k$ embed. Take such a structure $D$ and then set $M_{k+1}$ to be equal to $D$.
%
%
%
%
%
%\
%
%\
%
%
%
Now suppose that $k\equiv 1\mod{3}$. Let $L_1 = (A_{kj}, B_{kj}, f_{kj})_{j\in\mb{N}}$ be the list of all triples $(A,B,f)$ such that $(A,B)\in S$ and $f:A\to M_k$ is a monomorphism. This list is countable as $S$ is and there are finitely many monomorphisms from $A$ into $M_k$. 
Let $(i,j) \in [\mathbf{1}] \times \mathbb{N}$ be the unique pair satisfying $k = \beta_1(i,j)$. 
Then as $\beta_1(i,j)\geq i$, the map $f_{ij}:A_{ij}\to M_i\subseteq M_k$ exists. Therefore, we can use the MAP to define $M_{k+1}$ such that $M_k\subseteq M_{k+1}$ and the monomorphism $f_{ij}:A_{ij}\to M_k$ extends to some monomorphism $g_{ij}:B_{ij}\rarr{}M_{k+1}$ (see \autoref{existsmb} (1)). This ensures that every possible mono-amalgamation occurs. If $k\equiv 2\mod{3}$, let $L_2 = (P_{kj}, Q_{kj}, \bar{f}_{kj})_{j\in\mb{N}}$ be the list of all triples $(P,Q,\bar{f})$ such that $(P,Q)\in S$ and $\bar{f}:P\to M_k$ is an antimonomorphism; again, this list is countable. 
Let $(i,j) \in [\mathbf{2}] \times \mathbb{N}$ be the unique pair satisfying $k = \beta_2(i,j)$. 
Since $\beta_2(i,j)\geq i$ it follows that the map $\bar{f}_{ij}:P_{ij}\to M_i\subseteq M_k$ is defined. So we can use the AMAP to define $M_{k+1}$ such that $M_k\subseteq M_{k+1}$ and the antimonomorphism $\bar{f}_{ij}:P_{ij}\to M_k$ extends to some antimonomorphism $\bar{g}_{ij}:Q_{ij}\rarr{}M_{k+1}$ (see \autoref{existsmb} (2)). This construction ensures that every possible antimono-amalgamation occurs. 

\begin{figure}[h]
\centering
\begin{tikzpicture}[node distance=2cm,scale=0.9]
\begin{scope}[xshift=-3cm]
\node(D) at (0,1.5) {$M_{k+1}$};
\node(R) at (1.5,0) {$B_{ij}$};
\node(L) at (-1.5,0) {$M_k$};
\node(H) at (0,-1.5) {$A_{ij}$};
\path[->,thick,dotted] (R) edge node[above right] {$g_{ij}$} (D);
\path[right hook->,thick,dotted] (L) edge node[above left] {$\iota_k$} (D);
\path[right hook->] (H) edge node[below right] {$\iota_{ij}$} (R);
\path[->] (H) edge node[below left] {$f_{ij}$} (L);
\node(A) at (0,-2.5) {(1)};
\end{scope}

\begin{scope}[xshift=3cm]
\node(D) at (0,1.5) {$M_{k+1}$};
\node(R) at (1.5,0) {$Q_{ij}$};
\node(L) at (-1.5,0) {$M_k$};
\node(H) at (0,-1.5) {$P_{ij}$};
\path[->,thick,dotted] (R) edge node[above right] {$\bar{g}_{ij}$} (D);
\path[right hook->,thick,dotted] (L) edge node[above left] {$\iota_k$} (D);
\path[right hook->] (H) edge node[below right] {$\iota_{ij}$} (R);
\path[->] (H) edge node[below left] {$\bar{f}_{ij}$} (L);
\node(B) at (0,-2.5) {(2)};
\end{scope}

\end{tikzpicture}
\caption{Amalgamations performed in the proof of \autoref{mbfraisse}. Here, the $\iota$'s are inclusion mappings.}\label{existsmb}
\end{figure}
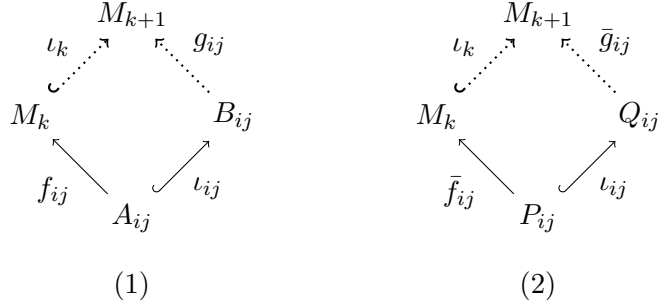

Following this, define $\mc{M} = \bigcup_{k\in\mb{N}}M_k$. All that remains to show is that $\mc{M}$ has age $\ms{C}$ and that $\mc{M}$ is MB-homogeneous. Our construction ensures that every isomorphism type of $\ms{C}$ appears at a 0 mod 3 stage, so every structure from $\ms{C}$ embeds into $\mc{M}$. Conversely, we have that $M_k\in\ms{C}$ for all $k\in\mb{N}$. As $\ms{C}$ is closed under substructures, every structure that embeds into $\mc{M}$ is in $\ms{C}$, showing that Age$(\mc{M}) = \ms{C}$. Now suppose that $A\subseteq B\in\ms{C}$ and $f:A\rarr{}\mc{M}$ is a monomorphism. As $Af$ is finite, it follows that there exists $j\in [\mathbf{1}]$ such that $Af\subseteq M_j$. Furthermore, there exists a triple $(A_{j\ell},B_{j\ell},f_{j\ell})\in L_1$ such that there exists an isomorphism $\theta:B\to B_{j\ell}$ with $A\theta|_A = A_{j,l}$ and $f = \theta|_Af_{j\ell}$. Define $n = \beta_1(j,\ell)$. Since $n \geq j$, it follows that $Af = A_{j\ell}f_{j\ell}\subseteq M_j \subseteq M_n$. Here, $M_{n+1}$ is constructed by monoamalgamating $M_n$ and $B_{j\ell}$ over $A_{j\ell}$. Therefore, this amalgamation provides an extension $g = \theta g_{j\ell}:B\to\mc{M}$ to the monomorphism $f$ (see \autoref{amaldiag} for a diagram). 

This proves that $\mc{M}$ has the MEP. Using a similar argument with $P\subseteq Q\in \ms{C}$ and an antimonomorphism $\bar{f}:P\to\mc{M}$, we can show that $\mc{M}$ has the AMEP. So $\mc{M}$ is MB-homogeneous by \autoref{mepbep}.
\end{proof}

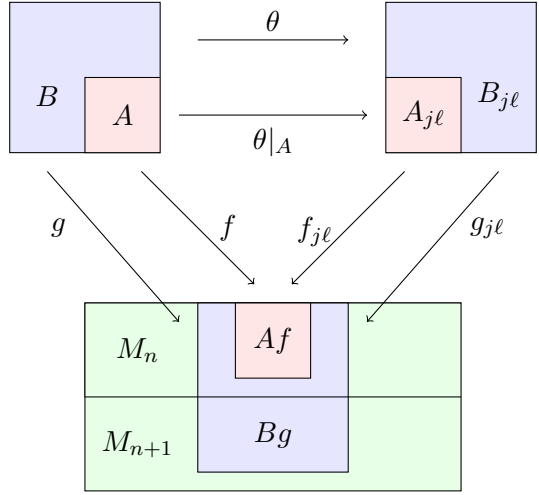
\begin{figure}
\centering
\begin{tikzpicture}[node distance=1.8cm,inner sep=0.7mm,scale=1]

%A'',B'',labels
\draw[fill=blue!10!white] (-3.5,3) rectangle (-1.5,1);
\draw[fill=red!10!white] (-2.5,2) rectangle (-1.5,1);
\node (a11) at (-2,1.5) {$A$};
\node (b11) at (-3,1.75) {$B$};

%A',B',labels
\draw[fill=blue!10!white] (3.5,3) rectangle (1.5,1);
\draw[fill=red!10!white] (2.5,2) rectangle (1.5,1);
\node (a1) at (2,1.5) {$A_{j\ell}$};
\node (b) at (3,1.75) {$B_{j\ell}$};

\draw[->] (-1.25,1.5) -- (1.25,1.5);
\node (e1) at (0,1.15) {$\theta|_A$};

\draw[->] (-1,2.5) -- (1,2.5);
\node (e) at (0,2.75) {$\theta$};

\draw[fill=green!10!white] (-2.5,-1) rectangle (2.5,-2.25);
\draw[fill=green!10!white] (-2.5,-1) rectangle (2.5,-3.5);
\node at (-1.8,-1.625) {$M_n$};
\node at (-1.8,-2.875) {$M_{n+1}$};
% alt2 A,B,labels in centre
\draw[fill=blue!10!white] (-1,-3.25) rectangle (1,-1);
\draw[fill=red!10!white] (-0.5,-2) rectangle (0.5,-1);
\node (a) at (0,-1.5) {$Af$};
\node (b) at (0,-2.75) {$Bg$};
\draw[<-] (0.25,-0.75) -- (1.75,0.75);
\node (alpha) at (0.55,-0.05) {$f_{j\ell}$};
\draw[<-] (1.25,-1.25) -- (3,0.75);
\node (alpha) at (2.85,0) {$g_{j\ell}$};
\draw[<-] (-0.25,-0.75) -- (-1.75,0.75);
\node (f) at (-0.6,0) {$f$};
\draw[<-] (-1.25,-1.25) -- (-3,0.75);
\node (h1alpha) at (-2.85,0) {$g$};
\draw (-2.5,-2.25) -- (2.5,-2.25);
\end{tikzpicture}
\caption{Diagram of maps in the proof of \autoref{mbfraisse}, with colours to illustrate. The map $g = \theta g_{j\ell}:B\to\mc{M}$ is an monomorphism extending $f$, proving that $\mc{M}$ has the MEP.}\label{amaldiag}
\end{figure}

An important consequence of Fra\"{i}ss\'{e}'s original theorem is the fact that any two Fra\"{i}ss\'{e} limits with the same age are isomorphic. While we cannot guarantee that any two structures with the same age constructed in the manner of \autoref{mbfraisse} are isomorphic (see \autoref{s5} for some examples) we can provide a uniqueness condition for this method of construction using a weaker notion of equivalence. Once again, we extend an idea of \cite{cameron2006homomorphism} to achieve this goal.

Let $\mc{M}$ and $\mc{N}$ be two countable structures with the same signature $\sigma$. We say that $\mc{M}$ and $\mc{N}$ are \emph{bi-equivalent} if:
\begin{itemize}
\item Age$(\mc{M}) =$ Age$(\mc{N})$, and;
\item every embedding from a finite structure of $\mc{M}$ into $\mc{N}$ extends to a bijective homomorphism $\alpha:\mc{M}\rarr{}\mc{N}$, and vice versa.
\end{itemize}

Note that bi-equivalence is an equivalence relation on structures with the same signature. Two $\sigma$-structures $\mc{M}$ and $\mc{N}$ can be bi-equivalent without being isomorphic; see \autoref{urexample} for more details. We now show this is the relevant equivalence relation for MB-homogeneity.

\begin{proposition}\label{biequiv} (1) Let $\mc{M},\mc{N}$ be two bi-equivalent structures. If $\mc{M}$ is MB-homogeneous, so is $\mc{N}$.

(2) If $\mc{M},\mc{N}$ are MB-homogeneous and Age$(\mc{M})=$ Age$(\mc{N})$, then $\mc{M}$ and $\mc{N}$ are bi-equivalent.
\end{proposition}

\begin{proof} (1) By \autoref{mepbep} it suffices to show that $\mc{N}$ has the MEP and AMEP. If $\mc{M}$ and $\mc{N}$ are bi-equivalent, they are certainly mono-equivalent in the sense of Proposition 4.2 of \cite{cameron2006homomorphism}; as $\mc{M}$ is MM-homogeneous, so is $\mc{N}$ (by the same result of \cite{cameron2006homomorphism}) and thus $\mc{N}$ has the MEP.

Suppose then that $A\subseteq B\in$ Age$(\mc{N})$ and there exists an antimonomorphism $\bar{f}:A\rarr{}A'\subseteq\mc{N}$. Note that $A$ need not be isomorphic to $A'$. As Age$(\mc{M})=$ Age$(\mc{N})$ there exists a copy $A''$ of $A'$ in $\mc{M}$. Fix an isomorphism $e:A'\to A''$ between the two. Therefore, $e$ is a isomorphism from a finite structure of $\mc{N}$ into $\mc{M}$. Since the two are bi-equivalent, we extend this to a bijective homomorphism $\alpha:\mc{N}\to\mc{M}$. This in turn induces a bijective antimonomorphism $ \bar{\alpha}:\mc{M}\rarr{}\mc{N}$ by \autoref{antimono}. Now, define an antimonomorphism $\bar{h} = \bar{f}e: A\to A''$; this is an antimonomorphism from $A$ into $\mc{M}$. Since $\mc{M}$ is MB-homogeneous, it has the AMEP by \autoref{mepbep} and so we extend $\bar{h}$ to an antimonomorphism $\bar{h}':B\to \mc{M}$; see \autoref{xyequivdiag} for a diagram of this process. Now, the map $\bar{h}'\bar{\alpha}:B\rarr{}\mc{N}$ is a antimonomorphism; we need to show it extends $\bar{f}$. So, for all $a\in A$, and using the facts that $\alpha$ extends $e$ and $\bar{h}'$ extends $\bar{h} = \bar{f}e$, we have that \[a\bar{f} = a\bar{f}\alpha\bar{\alpha} = a\bar{f}e\bar{\alpha} = a\bar{h}'\bar{\alpha}.\] Therefore $\mc{N}$ has the AMEP.

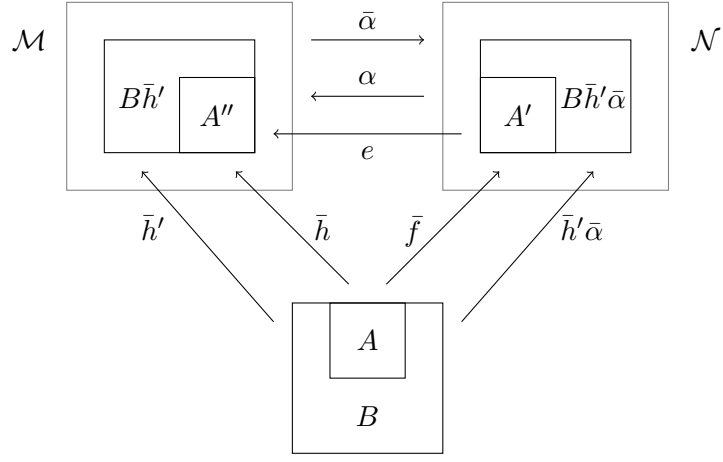
\begin{figure}
\centering
\begin{tikzpicture}[node distance=1.8cm,inner sep=0.7mm,scale=1]

% M,N,labels
\draw[black!50!white] (-4,0.5) rectangle (-1,3);
\draw[black!50!white] (1,0.5) rectangle (4,3);
\node(m) at (-4.5,2.5) {$\mc{M}$}; 
\node(n) at (4.5,2.5) {$\mc{N}$};

%A'',B'',labels
\draw (-3.5,2.5) rectangle (-1.5,1);
\draw (-2.5,2) rectangle (-1.5,1);
\node (a11) at (-2,1.5) {$A''$};
\node (b11) at (-3,1.75) {$B\bar{h}'$};

%A',B',labels
\draw (3.5,2.5) rectangle (1.5,1);
\draw (2.5,2) rectangle (1.5,1);
\node (a1) at (2,1.5) {$A'$};
\node (b) at (3,1.75) {$B\bar{h}'\bar{\alpha}$};

% arrows
\draw[->] (-0.75,2.5) -- (0.75,2.5);
\node (alpha) at (0,2.75) {$\bar{\alpha}$};

\draw[<-] (-0.75,1.75) -- (0.75,1.75);
\node (e1) at (0,2) {$\alpha$};

\draw[<-] (-1.25,1.25) -- (1.25,1.25);
\node (e) at (0,1) {$e$};

% alt2 A,B,labels in centre
\draw (-1,-3) rectangle (1,-1);
\draw (-0.5,-2) rectangle (0.5,-1);
\node (a) at (0,-1.5) {$A$};
\node (b) at (0,-2.5) {$B$};
\draw[->] (0.25,-0.75) -- (1.75,0.75);
\node (alpha) at (0.6,-0.05) {$\bar{f}$};
\draw[->] (1.25,-1.25) -- (3,0.75);
\node (alpha) at (2.85,0) {$\bar{h}'\bar{\alpha}$};
\draw[->] (-0.25,-0.75) -- (-1.75,0.75);
\node (f) at (-0.6,0) {$\bar{h}$};
\draw[->] (-1.25,-1.25) -- (-3,0.75);
\node (h1alpha) at (-2.85,0) {$\bar{h}'$};
\end{tikzpicture}
\caption{Diagram of maps in the proof of \autoref{biequiv}. The map $\bar{h}'\bar{\alpha}$ is an antimonomorphism extending $\bar{f}$, proving that $\mc{N}$ has the AMEP.}\label{xyequivdiag}
\end{figure}

(2) Let $f:A\rarr{}B$ be a bijective embedding from a finite structure of $\mc{M}$ into $\mc{N}$. We utilise a back and forth argument constructing the bijection over countably many steps; so set $A_0 = A, B_0 = B$ and $f_0 = f$. At some stage, we have a bijective monomorphism $f_k:A_k\rarr{}B_k$ extending $f$, where $A_i\subseteq A_{i+1}$ and $B_i\subseteq B_{i+1}$ for all $i\leq k$. As both $\mc{M}$ and $\mc{N}$ are countable, we can enumerate the points $M = \{m_0,m_1,\ldots\}$ and $N = \{n_0,n_1,\ldots\}$.

If $k$ is even, select a point $m_i\in\mc{M}\smallsetminus$ dom $f_k$, where $i$ is the smallest natural number such that $m_i\notin$ dom $f_k$. We have that $A_k\cup\{m_i\}\subseteq\mc{M}$ and so is an element of Age$(\mc{N})$ by assumption. As $\mc{N}$ is MB-homogeneous it has the MEP and so $f_k$ can be extended to a monomorphism $f'_{k+1}:A_k\cup\{m_i\}\rarr{}\mc{N}$. Restricting the codomain to the image of $f_{k+1}'$ gives a bijective monomorphism $f_{k+1}:A_k\cup\{m_i\}\rarr{}B_k\cup\{m_if'_{k+1}\}$ extending $f_k$. If $k$ is odd, note that $\bar{f}_k:B_k\rarr{}A_k$ is a bijective antimonomorphism. Select a point $n_j\in\mc{N}\smallsetminus$ dom $\bar{f}_k$, where $j$ is the least natural number such that $n_j\notin$ im $f_k$. Therefore $B_k\cup\{n_j\}\subseteq \mc{N}$ and is an element of Age$(\mc{M})$. As $\mc{M}$ is MB-homogeneous, use the AMEP to extend $\bar{f}_k$ to an antimonomorphism $\bar{f}'_{k+1}:B_k\cup\{n_j\}\rarr{}\mc{M}$. Restricting the codomain to the image once more delivers the required bijective antimonomorphism $\bar{f}_{k+1}:B_k\cup\{n_j\}\rarr{}A_k\cup\{n_j\bar{f}'_{k+1}\}$ extending $\bar{f}_k$. By \autoref{antimono}, $f_{k+1}$ is a bijective homomorphism extending $f_k$ as required. Repeating this process countably many times, ensuring that each $m_i$ is in the domain and each $n_j$ in the image, provides a bijective homomorphism $\alpha:\mc{M}\rarr{}\mc{N}$ extending $f$. The converse direction is entirely analogous.
\end{proof}
 
We conclude the section by observing that due to \autoref{mbfraisse} and \autoref{biequiv}, a MB-homogeneous structure is determined by its age up to bi-equivalence.

\section{MB-homogeneous graphs}\label{s5}

In this section, we focus on MB-homogeneous graphs in more detail. Here, a \emph{graph} $\G$ is a set of \emph{vertices} $V\G$ together with a set of \emph{edges} $E\G$, where this edge set interprets a irreflexive and symmetric binary relation $E$. If $(u,v)\in E\G$, we say that $u$ and $v$ are \emph{adjacent} and write $u\sim_\G v$. 
When it is clear from context in which graph we are working, we will omit the subscript and simply write $u \sim v$. 
We define the \emph{neighbourhood} $N(v)$ of a vertex $v\in V\G$ by $N(v) = \{u\in V\G\; :\; u\sim_\G v\}$. If $(u,v)\notin E\G$, we say that $(u,v)$ is a \emph{non-edge}, that they are \emph{non-adjacent}, and write $u\nsim_\G v$. We say that a vertex $v\in V\G$ is \emph{independent} of a set of vertices $U\subseteq V\G$ if $v\nsim_\G u$ for all $u\in U$. For a graph $\G$, define the \emph{complement} $\bar{\G}$ of $\G$ to be the graph with vertex set $V\bar{\G} = V\G$ and edges given by $u\sim_{\bar{\G}} v$ if and only if $u\nsim_{\G} v$. For $n\in \mb{N}\cup\{\aleph_0\}$, recall that a \emph{complete graph $K^n$ on $n$ vertices} is a graph on vertex set $VK^n$ (with $|VK^n| = n$) with edges given by $u\sim v\in K^n$ if and only if $u\neq v\in VK^n$. The complement of the graph $K^n$ is called the \emph{null graph on $n$ vertices} or an \emph{independent set on $n$ vertices}, and is denoted by $\bar{K}^n$. Say that $\Delta$ is an \emph{induced subgraph} of a graph $\G$ if $V\Delta \subseteq V\G$ and $u\sim_\G v$ if and only if $u\sim_\Delta v$. A graph $\Delta$ is a \emph{spanning subgraph} of $\G$ if $V\Delta = V\G$ and $u\sim_\Delta v$ implies that $u\sim_\G v$. 

A logical place to start our investigation is by demonstrating some properties of MB-homogeneous graphs. 

\begin{proposition}\label{complement} Let $\G$ be an MB-homogeneous graph. Then its complement $\bar{\G}$ is also MB-homogeneous.
\end{proposition}

\begin{proof} We note that $A\in$ Age$(\G)$ if and only if $\bar{A}\in$ Age$(\bar{\G})$. Since $\G$ is MB-homogeneous, it has the MEP and AMEP by \autoref{mepbep}. Now suppose that $A\subseteq B\in$ Age$(\G)$ and that $\bar{f}:A\rarr{}\G$ is an antimonomorphism. Any such $\bar{f}$ preserves non-edges and may change edges to non-edges; so $\bar{f}:\bar{A}\rarr{}\bar{\G}$ is a monomorphism. As $\G$ has the AMEP, $\bar{f}$ can be extended to an antimonomorphism $\bar{g}:B\rarr{}\G$; this in turn induces a monomorphism $\bar{g}:\bar{B}\rarr{}\bar{\G}$ and hence $\bar{\G}$ has the MEP. The proof that $\bar{\G}$ has the AMEP is similar.
\end{proof}

Following this, we can guarantee that certain subgraphs appear in an MB-homogeneous graph. Cameron and Ne\v{s}et\v{r}il \cite[Proposition 2.5]{cameron2006homomorphism} prove that every infinite, non-null MM-homogeneous graph must contain $K^{\aleph_0}$ as an induced subgraph; we expand this proposition.

\begin{corollary}\label{infcompnull} Any infinite non-complete, non-null MB-homogeneous graph $\G$ contains both $K^{\aleph_0}$ and $\bar{K}^{\aleph_0}$ as induced subgraphs.
\end{corollary}

\begin{proof} Any MB-homogeneous graph is necessarily MM-homogeneous and hence it contains $K^{\aleph_0}$ as an induced subgraph from the aforementioned result of \cite{cameron2006homomorphism}. By \autoref{complement} we have that $\bar{\G}$ is also MB-homogeneous and so contains $K^{\aleph_0}$ as an induced subgraph; the result follows from this.
\end{proof}

A consequence of this argument is that an MB-homogeneous graph $\G$ is neither a locally finite graph (where for all $v\in V\G$, there are only finitely many edges involving $v$) nor the complement of a locally finite graph. In fact, we can say more than this. Recall that the \emph{distance} between two vertices $u$ and $v$ in a connected graph $\G$ is the length of the shortest path between $u$ and $v$, and that the \emph{diameter} of $\G$ is the greatest distance between any two vertices of $\G$. The next result is a restatement of \cite[Proposition 1.1(c)]{cameron2006homomorphism}; the proof follows as every MB-homogeneous graph is also an MH-homogeneous graph.

\begin{corollary}[Proposition 1.1(c), \cite{cameron2006homomorphism}] Suppose that $\G$ is a connected MB-homogeneous graph. Then $\G$ has diameter at most 2 and every edge is contained in a triangle.\dne
\end{corollary}

We now examine cases where the graph is disconnected (and non-null to avoid triviality). It is shown in \cite{cameron2006homomorphism} that any disconnected MH-homogeneous graph is a disjoint union of complete graphs of all the same size. We use this result in conjunction with \autoref{infcompnull} to see that the only candidates for a disconnected MB-homogeneous graph must be disjoint unions of infinite complete graphs. Building on an observation of \cite{rusinov2010homomorphism} that any disconnected MM-homogeneous graph is a disjoint union of infinite complete graphs, we classify disconnected MB-homogeneous graphs in our next result.

\begin{proposition}\label{dunionkgr} Let $\G = \bigsqcup_{i\in I}K_i$, where $K_i\cong K^{\aleph_0}$ for all $i$ in some index set $I$.

(1) If $I$ is finite with size $n>1$, then $\G$ is MM-homogeneous but not MB-homogeneous.

(2) If $I$ is countably infinite, then $\G$ is MB-homogeneous.
\end{proposition}

\begin{proof} In both cases, we note that every $A\in$ Age$(\G)$ can be decomposed as finite disjoint union of finite complete graphs; so we can write that $A = \bigsqcup_{j=1}^k C_j$, where $C_j$ is a complete graph of some finite size.

(1) As $|I|=n$ we note that $k\leq n$ for all $A=\bigsqcup_{j=1}^k C_j$ in the age of $\G$. Suppose that $A\subseteq B\in$ Age$(\G)$ with $B\smallsetminus A = \{b\}$. We have two choices for $b$; either $b$ is completely independent of $A$ or $b$ is related to exactly one $C_j$ for some $1\leq j\leq k$. If it is the former, we can extend any monomorphism $f:A\rarr{}\G$ to a monomorphism $g:B\rarr{}\G$ by sending $b$ to any vertex $v\in V\G\smallsetminus Af$. If it is the latter, then we can extend any monomorphism $f:A\rarr{}\G$ to a monomorphism $g:B\rarr{}\G$ by sending $b$ to a vertex $v\in K_i\smallsetminus C_jf$, where $j$ is defined as above. Hence $\G$ has the MEP and so is MM-homogeneous. However, note that $\G$ does not embed an independent $n+1$-set and so $\G$ is not MB-homogeneous by \autoref{infcompnull}.

(2) The proof that $\G$ is MM-homogeneous when $I$ is infinite is as above. Assume then that $A\subseteq B\in$ Age$(\G)$ with $B\smallsetminus A = \{b\}$, and let $\bar{f}:A\rarr{}\G$ be an antimonomorphism. We note that as $A$ is finite then $A\bar{f}$ is finite; since $I$ is infinite, there will always exist $i\in I$ such that $K_i\;\cap A\bar{f} = \emptyset$. Therefore, regardless of how $b$ is related to $A$, we can extend $\bar{f}$ to an antimonomorphism $\bar{g}:B\rarr{}\G$ by mapping $b$ to some $v\in K_i$, where $i$ is as stated above. Hence $\G$ is MB-homogeneous by \autoref{mepbep}.
\end{proof}

\begin{remark} We note that from \autoref{dunionkgr} and \autoref{complement} that the complement of $\bigsqcup_{i\in\mb{N}}K_i^{\aleph_0}$, which is the complete multipartite graph with infinitely many partitions each of infinite size, is also MB-homogeneous.
\end{remark}

The proof that $\G = \bigsqcup_{i\in\mb{N}}K_i^{\aleph_0}$ is MB-homogeneous relied on the existence of an independent element to every image of a finite antimonomorphism. Our aim now is to obtain some sufficient conditions for MB-homogeneity along these lines in order to construct some new examples.

\begin{definition}\label{deftrtf} Let $\G$ be an infinite graph.
\begin{itemize}
\item Say that $\G$ has \emph{property $(\triangle)$} if for every finite set $U\subseteq V\G$ there exists $u\in V\G$ such that $u$ is adjacent to every member of $U$.
\item Say that $\G$ has \emph{property $(\therefore)$} if for every finite set $V\subseteq V\G$ there exists $v\in V\G$ such that $v$ is non-adjacent to every member of $V$.
\end{itemize}
(See \autoref{figtrtf} for a diagram.)
\end{definition}

\begin{figure}[h]
\centering
\begin{tikzpicture}[node distance=1.8cm,inner sep=0.7mm,scale=0.8]
\title{Intermediate monoids of a first order structure}
\draw (-5,-3) rectangle (5,2.5);
\node(C) at (0,-2.5) {$\G$};
\node(tr) at (-2.5,2) {property $\tr$};
\node(U5) at (-1.5,0) [circle,draw] {};
\node(U4) at (-2,0) [circle,draw] {};
\node(U3) at (-2.5,0) [circle,draw] {};
\node(U2) at (-3,0) [circle,draw] {};
\node(U1) at (-3.5,0) [circle,draw] {};
\node(U) at (-4,1) {$U$};
\draw(U3) ellipse (1.8cm and 0.5cm);
\node(x) at (-2.5,-2) [circle,draw,label=below:$u$] {};
\draw (x) -- (U5);
\draw (x) -- (U4);
\draw (x) -- (U3);
\draw (x) -- (U2);
\draw (x) -- (U1);
\node(tr) at (2.5,2) {property $\tf$};
\node(V1) at (1.5,0) [circle,draw] {};
\node(V2) at (2.5,0) [circle,draw] {};
\node(V3) at (3.5,0) [circle,draw] {};
\node(V) at (4,1) {$V$};
\draw(V2) ellipse (1.8cm and 0.5cm);
\node(v) at (2.5,-2) [circle,draw,label=below:$v$] {};
\draw[dotted] (v) -- (V1);
\draw[dotted] (v) -- (V2);
\draw[dotted] (v) -- (V3);
\end{tikzpicture}
\caption{A diagram of \autoref{deftrtf}}\label{figtrtf}
\end{figure}

We note here that a graph $\G$ has property $\tr$ if and only if $\G$ is \emph{primitive positive (pp)-closed} (also \emph{algebraically closed} as defined in \cite{dolinka2014automorphism}). Due to the self-complementary nature of these properties, if $\G$ has property $\tf$ then its complement $\G$ is a pp-closed graph. These properties prove to be sufficient for MB-homogeneity.

\begin{proposition}\label{trtfmb} Let $\G$ be an infinite graph. If $\G$ has both properties $\tr$ and $\tf$ then $\G$ is MB-homogeneous.
\end{proposition}

\begin{proof} Suppose that $A\subseteq B\in$ Age$(\G)$ with $B\smallsetminus A = \{b\}$, and that $f:A\rarr{}\G$ is a monomorphism. As $A$ is finite, $Af$ is a finite set of vertices in $\G$ and so by property $\tr$ there exists a vertex $v$ of $\G$ such that $v$ is adjacent to every element of $Af$. We have then that $v$ is a potential image point of $b$, and the map $g:B\rarr{}\G$ extending $f$ and sending $b$ to $v$ is a monomorphism; so $\G$ has the MEP. Using property $\tf$ in a similar fashion shows that $\G$ has the AMEP and so is MB-homogeneous by \autoref{mepbep}.
\end{proof}

\begin{remark} The converse of this result is not true. \autoref{dunionkgr} (2) is an example of an MB-homogeneous graph with property $\tf$ but not property $\tr$. Its complement is an example of an MB-homogeneous graph with property $\tr$ but not property $\tf$.
\end{remark}

This shows that any pp-closed graph $\G$ whose complement is also pp-closed is MB-homogeneous. We now present a notion of equivalence that extends an idea of \cite{dolinka2014automorphism}.

\begin{definition}\label{defbieq} Let $\G,\Delta$ be two graphs. We say that $\G$ and $\Delta$ are \emph{equivalent up to bimorphisms} if there exist bimorphisms $\alpha:\G\rarr{}\Delta$ and $\beta:\Delta\rarr{}\G$.
\end{definition}

\begin{remark}
Informally, this definition says that you can start with $\G$, and draw some number (possibly infinite) of extra edges onto $\G$ to get $\Delta$. From there, you can draw some number (possibly infinite) of extra edges onto $\Delta$, to get a graph isomorphic to $\G$. Note that this definition is equivalent to saying that $\G$ and $\Delta$ contain each other as spanning subgraphs.
\end{remark}

This is a weaker version of bi-equivalence introduced in \autoref{biequiv}. Every pair of bi-equivalent graphs are equivalent up to bimorphisms by definition, but the converse is not true; specifically, two graphs that are equivalent up to bimorphisms need not have the same age (see \autoref{beqr} and \autoref{urexample}).  Justifying the name, this is an equivalence relation of graphs (up to isomorphism) which we denote by $\sim_b$. The product $\alpha\beta:\G\rarr{}\G$ of the two bijective homomorphisms induces a bimorphism of $\G$, and $\alpha\beta$ is an automorphism if and only if $\alpha:\G\to \Delta$ and $\beta:\Delta\to\G$ are bijective isomorphisms. If Bi$(\G) =$ Aut$(\G)$, then $[\G]^{\sim_b}$ is a singleton equivalence class. We show that equivalence up to bimorphisms preserves properties $\tr$ and $\tf$.

\begin{proposition}\label{bieqtrtf} Let $\G,\Delta$ be graphs that are equivalent up to bimorphisms via $\alpha:\G\rarr{}\Delta$ and $\beta:\Delta\rarr{}\G$. Then $\G$ has properties $\tr$ and $\tf$ if and only if $\Delta$ does.
\end{proposition}

\begin{proof} 
Assume that $\G$ has properties $\tr$ and $\tf$, and let $Y$ be any subset of $\Delta$. As $\alpha$ is bijective, $Y\bar{\alpha}$ is a finite subset of $\G$; so by property $\tr$ there exists a vertex $x\in \G$ that is adjacent to every element of $Y\bar{\alpha}$. Since $\alpha$ is a monomorphism, $x\alpha$ is adjacent to every element of $Y$ and hence $\Delta$ has property $\tr$. Similarly, $Y\beta$ is a finite subset of $\G$, so property $\tf$ implies that there is $w\in \G$ independent of $Y\beta$. As $\bar{\beta}$ is an antimonomorphism it preserves non-edges and so $w\bar{\beta}$ is a vertex of $\Delta$ independent of $Y$. Therefore, $\Delta$ has property $\tf$. The converse direction is symmetric.
\end{proof}

In fact, we can say more about the connection between equivalence up to bimorphisms and properties $\tr$ and $\tf$. 

\begin{proposition}\label{trtfbieq} If $\G,\Delta$ are two graphs with properties $\tr$ and $\tf$, then $\G$ and $\Delta$ are equivalent up to bimorphisms.
\end{proposition}

\begin{proof}
Assume that $\G,\Delta$ are two graphs with properties $\tr$ and $\tf$. We use a back and forth argument to construct a bijective homomorphism $\alpha:\G\to\Delta$ and a bijective antimonomorphism $\bar{\beta}:\G\to\Delta$, which by \autoref{antimono} will be the inverse of a bijective homomorphism $\beta:\Delta\to\G$. As both $\G$ and $\Delta$ are countable, we can enumerate their vertices as $V\G = \{c_0,c_1,\ldots\}$ and $V\Delta = \{d_0,d_1,\ldots\}$. Let $C_0 = \{c_0\}$ and $D_0 = \{d_0\}$, and set $f_0:C_0 \to D_0$ be the map sending $c_0$ to $d_0$; this is a bimorphism. Assume now that we have extended $f$ to a bimorphism $f_k:C_k\to D_k$, where $C_i$ and $D_i$ are finite and $C_i\subseteq C_{i+1}$ and $D\subseteq D_{i+1}$ for all $0\leq i \leq {k-1}$.

If $k$ is even, select the vertex $c_j\in \G$ where $j$ is the smallest number such that $c_j\notin C_k$. As $\Delta$ has property $\tr$, there exists a vertex $u\in\Delta$ such that $u$ is adjacent to every element of $D_k$. Define a map $f_{k+1}:C_k\cup\{c_j\}\to D_k\cup\{u\}$ sending $c_j$ to $u$ and extending $f$; this map is a bijective homomorphism as any edge from $c_j$ to some element of $C_k$ is preserved.

Now, if $k$ is odd, choose the vertex $d_j\in\Delta$ where $j$ is the smallest number such that $d_j\notin D_k$. As $\G$ has property $\tf$, there exists a vertex $v\in\G$ such that $v$ is independent of every element of $C_k$. Define a map $f_{k+1}:C_k\cup\{v\}\to D_k\cup\{d_j\}$ sending $v$ to $d_j$ and extending $f_{k}$. Then $f_{k+1}$ is a bijective homomorphism, since $f_k$ is a bijective homomorphism and every edge between $c$ and $C_k$ is preserved; because there are none.

Repeating this process infinitely many times, ensuring that each vertex of $\G$ appears at an even stage and each vertex of $\Delta$ appears at an odd stage, defines a bijective homomorphism $\alpha:\G\to\Delta$. We can construct a bijective antimonomorphism $\bar{\beta}:\G\to\Delta$ in a similar fashion by replacing homomorphism with antimonomorphism and using property $\tf$ of $\Delta$ at even steps and property $\tr$ of $\G$ at odd steps. So the converse map $\beta:\Delta\to\G$ is a bijective homomorphism and so $\G$ and $\Delta$ are equivalent up to bimorphisms.
\end{proof}

Recall that the \emph{random graph} $R$ is the countable universal homogeneous graph. It is well known (see \cite{cameron1997random}) that $R$ is characterised up to isomorphism by the following extension property (EP):

\begin{quote} (EP) For any two finite disjoint sets of vertices $U,V$ of $R$ there exists $x\in VR$ such that $x$ is adjacent to every vertex in $U$ and non-adjacent to any vertex in $V$. (see \autoref{epinr}.)
\end{quote}  

\begin{center}
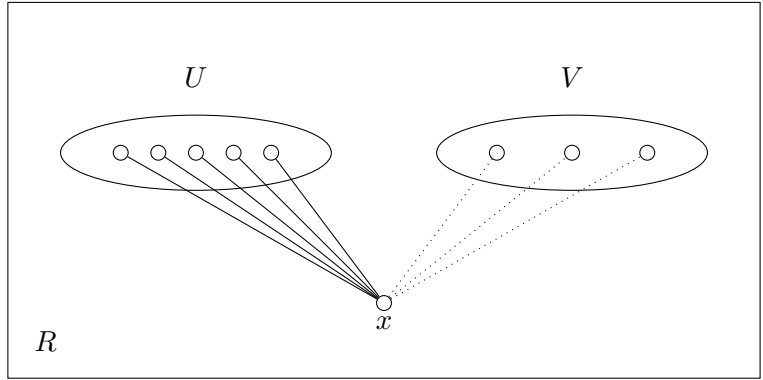
\begin{figure}[h]
\centering
\begin{tikzpicture}[node distance=1.8cm,inner sep=0.7mm,scale=1]
\node(R) at (-4.5,-2.5) {$R$};
\node(C) at (0,0) {};
\node(U5) at (-1.5,0) [circle,draw] {};
\node(U4) at (-2,0) [circle,draw] {};
\node(U3) at (-2.5,0) [circle,draw] {};
\node(U2) at (-3,0) [circle,draw] {};
\node(U1) at (-3.5,0) [circle,draw] {};
\node(V1) at (1.5,0) [circle,draw] {};
\node(V2) at (2.5,0) [circle,draw] {};
\node(V3) at (3.5,0) [circle,draw] {};
\node(U) at (-2.5,1) {$U$};
\node(V) at (2.5,1) {$V$};
\node(x) at (0,-2) [circle,draw,label=below:$x$] {};
\draw(U3) ellipse (1.8cm and 0.5cm);
\draw(V2) ellipse (1.8cm and 0.5cm);
\draw (-5,-3) rectangle (5,2);
\draw (x) -- (U5);
\draw (x) -- (U4);
\draw (x) -- (U3);
\draw (x) -- (U2);
\draw (x) -- (U1);
\draw[dotted] (x) -- (V1);
\draw[dotted] (x) -- (V2);
\draw[dotted] (x) -- (V3);
\end{tikzpicture}
\caption{(EP) in the random graph $R$}\label{epinr}
\end{figure}
\end{center}
\vspace{-0.8em}

The next result extends \cite[Proposition 2.1(i)]{cameron2006homomorphism} and establishes a complementary condition to the graph case of \cite[Corollary 2.2]{dolinka2014automorphism}. 

\begin{corollary}\label{beqr} Suppose that $\G$ is a countable graph. Then $\G$ has properties $\tr$ and $\tf$ if and only if $\G$ and $R$ are equivalent up to bimorphisms.
\end{corollary}

\begin{proof} 
As $R$ has both properties $\tr$ and $\tf$, the converse direction follows from \autoref{bieqtrtf}, and the forward direction follows from \autoref{trtfbieq}.
\end{proof}

\begin{remark} These three results together show that the equivalence class $[R]^{\sim_b}$ is precisely the set of all countable graphs $\G$ with properties $\tr$ and $\tf$.
\end{remark}

\subsection{Constructing uncountably many examples}

The aim of this subsection is to use \autoref{trtfmb} in order to construct uncountably many MB-homogeneous graphs.

\begin{definition}\label{urexample} Let $P=(p_n)_{n\in\mb{N}_0}$ be an infinite binary sequence. Define the graph $\G(P)$ on the infinite vertex set $V\G(P) = \{v_0,v_1,\ldots\}$ with edge relation $v_i \sim v_j$ if and only if $p_{\text{max}(i,j)}=0$. Say that $\G(P)$ is the graph \emph{determined} by the binary sequence $P$. Furthermore, denote the graph induced by $\G(P)$ on any subset $VX$ of $V\G$ by $\G(X)$.
\end{definition}

If $\G(P)$ is such a graph, we observe that:
\begin{itemize}
\item if $p_i = 0$ then $v_i\sim v_j$ for all natural numbers $j<i$;
\item if $p_i = 1$ then $v_i\nsim v_j$ for all $j<i$;
\end{itemize}
where $<$ is the natural ordering on $\mb{N}$.  An example (where $P = (0,1,0,1,\ldots)$) is given in \autoref{evensodds}.

\begin{center}
\begin{figure}[h]
\centering
\begin{tikzpicture}[node distance=2cm,inner sep=0.7mm,scale=1.5]
\node (0) at (-4,0) [circle,draw,label=below:$v_0$] {};
\node (1) at (-3,0) [circle,draw,label=below:$v_1$] {};
\node (2) at (-2,0) [circle,draw,label=below:$v_2$] {};
\node (3) at (-1,0) [circle,draw,label=below:$v_3$] {};
\node (4) at (0,0) [circle,draw,label=below:$v_4$] {};
\node (5) at (1,0) [circle,draw,label=below:$v_5$] {};
\node (6) at (2,0) [circle,draw,label=below:$v_6$] {};
\node (7) at (3,0) [circle,draw,label=below:$v_7$] {};
\node (8) at (4,0) [circle,draw,label=below:$v_8$] {};
\node (9) at (5,0) {};
\draw (0) to [out=35,in=145] (2);
\draw (1) -- (2);
\draw (0) to [out=325,in=215] (4);
\draw (1) to [out=325,in=215] (4);
\draw (2) to [out=325,in=215] (4);
\draw (3) -- (4);
\draw (0) to [out=35,in=145] (6);
\draw (1) to [out=35,in=145] (6);
\draw (2) to [out=35,in=145] (6);
\draw (3) to [out=35,in=145] (6);
\draw (4) to [out=35,in=145] (6);
\draw (5) -- (6);
\draw (0) to [out=325,in=215] (8);
\draw (1) to [out=325,in=215] (8);
\draw (2) to [out=325,in=215] (8);
\draw (3) to [out=325,in=215] (8);
\draw (4) to [out=325,in=215] (8);
\draw (5) to [out=325,in=215] (8);
\draw (6) to [out=325,in=215] (8);
\draw (7) -- (8);
\draw [thick, loosely dotted] (8) -- (9);
\end{tikzpicture}
\caption{$\G(P)$, with $P = (0,1,0,1,\ldots)$}\label{evensodds}
\end{figure}
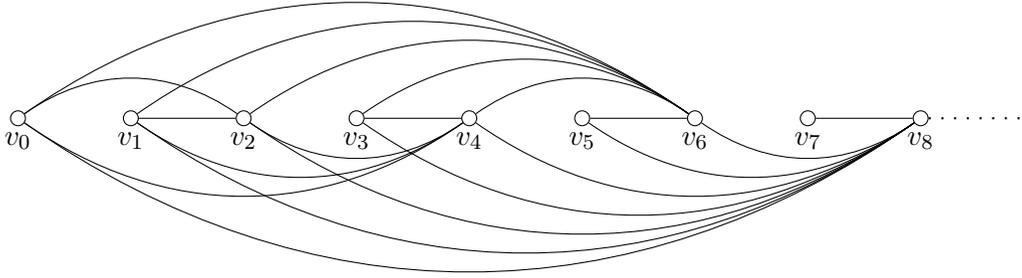
\end{center}
\vspace{-0.8cm}

If $P$ has infinitely many 0's and infinitely many 1's, then $\G(P)$ is an example of an MB-homogeneous graph. This is demonstrated in the following lemma, 

\begin{lemma}\label{gpmb}
Let $P=(p_n)_{n\in\mb{N}_0}$ be a binary sequence with infinitely many 0's and infinitely many 1's, and let $\G(P)$ be the graph determined by $P$. Then:
\begin{enumerate}[(1)]
 \item $\G(P)$ has properties $\tr$ and $\tf$ and is therefore MB-homogeneous. 
 \item $\G(P)$ is not homogeneous.
\end{enumerate}

\end{lemma}

\begin{proof}
\begin{enumerate}[(1)]
 \item As $P$ has infinitely many of each term, we have that for every finite subsequence $A = \{a_{i_1},\ldots,a_{i_k}\}$ of $P$ there exist natural numbers $c,d>i_k$ such that $p_c = 0$ and $p_d = 1$. This, together with the manner of the construction in \autoref{urexample}, ensures that $\G(P)$ has both properties $\tr$ and $\tf$. Therefore $\G(P)$ is MB-homogeneous by \autoref{trtfmb}. 
 \item 
We check that for any sequence $P$ satisfying the above conditions the graph $\G(P)$ is not isomorphic to a homogeneous graph by using the classification of countable homogeneous graphs by Lachlan and Woodrow \cite{lachlan1980countable}. As $\G(P)$ has both an edge and a non-edge it is neither $K^{\aleph_0}$ nor its complement $\bar{K}^{\aleph_0}$. Since $\G(P)$ is connected, it is not a disjoint union of complete graphs; as it has property $\tf$, it is not the complement of a disjoint union of complete graphs. From \autoref{infcompnull}, as $\G(P)$ is neither complete nor null it contains both $K^{\aleph_0}$ and $\bar{K}^{\aleph_0}$ as induced subgraphs; so it is not isomorphic to the $K_n$-free graph $H_n$ or its complement $\bar{H}_n$ for any $n\geq 3$. Finally, we note that as no term $p$ of $P$ can be both 0 and 1 simultaneously, there is no vertex $v_p$ that is adjacent to $\{v_0\}$ and non-adjacent to $\{v_1\}$. Hence $\G(P)$ does not satisfy the extension property characteristic of the random graph. This accounts for all countable graphs in the classification; so $\G$ is not a homogeneous graph.
\end{enumerate}

\end{proof}

\begin{remarks} Let $P,Q$ be two binary sequences with infinitely many 0's and infinitely many 1's. Any such infinite binary sequence contains every finite binary sequence $X$ as a subsequence. The finite induced subgraphs $\G(X)$ of $\G(P)$ are those induced on $V\G(X)$ by the edge relation of $P$. As this is true for all such binary sequences $P$ and $Q$, we conclude that $\G(P)$ and $\G(Q)$ have the same age. It can be shown from here that for any two such sequences $P$ and $Q$, then $\G(P)$ and $\G(Q)$ are bi-equivalent (see \autoref{biequiv}).

Throughout the rest of this section, any binary sequences $P,Q$ have infinitely many 0's and infinitely many 1's. This guarantees that any graph $\G(P)$ determined by $P$ has properties $\tr$ and $\tf$.
\end{remarks}

Our first lemma establishes a convention for the zeroth place of such a sequence. 

\begin{lemma}\label{firstplace} Suppose $P=(p_n)_{n\in\mb{N}_0}$ and $Q=(q_n)_{n\in\mb{N}_0}$ are binary sequences, and let $\G(P)$ and $\G(Q)$ be the graphs determined by $P$ and $Q$. If $p_i = q_i$ for all $i>0$, then $\G(P)\cong \G(Q)$. \dne
\end{lemma}

Following this, we can take $p_0 = p_1$ for any binary sequence $P$ without loss of generality; we adopt this as convention for the rest of the section. 

% Neighbourhoods and vertices.

Many natural questions arise from this construction. Perhaps the most pertinent of these is: to what extent does $\G(P)$ depend on the binary sequence $P$? Answering this question will tell us exactly how many new MB-homogeneous graphs there are of this kind. 

\begin{definition}\label{io}
  If $P$ is an infinite binary sequence, denote the $k^{\text{th}}$ consecutive string of 0's and 1's by $O_k$ and $I_k$ respectively. If $\G(P)$ is a graph determined by $P$, denote the vertex sets corresponding to these subsequences by $VO_k$ and $VI_k$. 
\end{definition}

See \autoref{diagramgraph} for an example of \autoref{io}.

\begin{center}
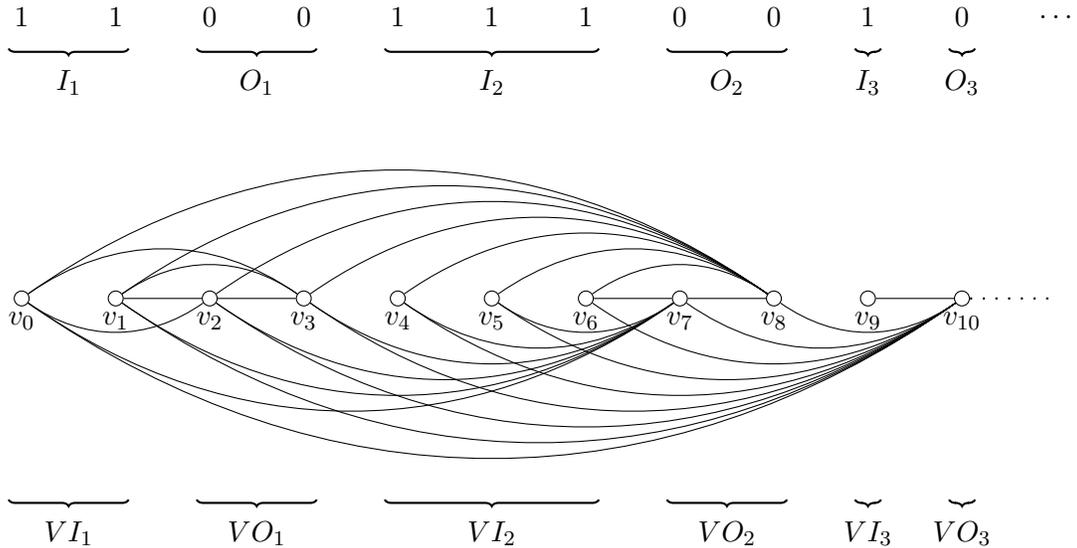
\begin{figure}[h]
\centering
\begin{tikzpicture}[node distance=2cm,inner sep=0.7mm,scale=1.25]
\node (0) at (-5,0) [circle,draw,label=below:$v_0$] {};
\node (1) at (-4,0) [circle,draw,label=below:$v_1$] {};
\node (2) at (-3,0) [circle,draw,label=below:$v_2$] {};
\node (3) at (-2,0) [circle,draw,label=below:$v_3$] {};
\node (4) at (-1,0) [circle,draw,label=below:$v_4$] {};
\node (5) at (0,0) [circle,draw,label=below:$v_5$] {};
\node (6) at (1,0) [circle,draw,label=below:$v_6$] {};
\node (7) at (2,0) [circle,draw,label=below:$v_7$] {};
\node (8) at (3,0) [circle,draw,label=below:$v_8$] {};
\node (9) at (4,0) [circle,draw,label=below:$v_9$] {};
\node (10) at (5,0) [circle,draw,label=below:$v_{10}$] {};
\node (11) at (6,0) {};
\node (0a) at (-5,3) {1};
\node (1a) at (-4,3) {1};
\node (2a) at (-3,3) {0};
\node (3a) at (-2,3) {0};
\node (4a) at (-1,3) {1};
\node (5a) at (0,3) {1};
\node (6a) at (1,3) {1};
\node (7a) at (2,3) {0};
\node (8a) at (3,3) {0};
\node (9a) at (4,3) {1};
\node (10a) at (5,3) {0};
\node (11a) at (6,3) {$\ldots$};
\node (0b) at (-5.2,2.8) {};
\node (1b) at (-3.8,2.8) {};
\node (2b) at (-3.2,2.8) {};
\node (3b) at (-1.8,2.8) {};
\node (4b) at (-1.2,2.8) {};
\node (6b) at (1.2,2.8) {};
\node (7b) at (1.8,2.8) {};
\node (8b) at (3.2,2.8) {};
\node (9b) at (3.8,2.8) {};
\node (91b) at (4.2,2.8) {};
\node (10b) at (4.8,2.8) {};
\node (101b) at (5.2,2.8) {};
\node (1c) at (-4.5,2.3) {$I_1$};
\node (3c) at (-2.5,2.3) {$O_1$};
\node (9c) at (0,2.3) {$I_2$};
\node (10c) at (2.5,2.3) {$O_2$};
\node (9c) at (4,2.3) {$I_3$};
\node (10c) at (5,2.3) {$O_3$};
\node (-0b) at (-5.2,-2) {};
\node (-1b) at (-3.8,-2) {};
\node (-2b) at (-3.2,-2) {};
\node (-3b) at (-1.8,-2) {};
\node (-4b) at (-1.2,-2) {};
\node (-6b) at (1.2,-2) {};
\node (-7b) at (1.8,-2) {};
\node (-8b) at (3.2,-2) {};
\node (-9b) at (3.8,-2) {};
\node (-91b) at (4.2,-2) {};
\node (-10b) at (4.8,-2) {};
\node (-101b) at (5.2,-2) {};
\node (1c) at (-4.5,-2.5) {$VI_1$};
\node (3c) at (-2.5,-2.5) {$VO_1$};
\node (9c) at (0,-2.5) {$VI_2$};
\node (10c) at (2.5,-2.5) {$VO_2$};
\node (9c) at (4,-2.5) {$VI_3$};
\node (10c) at (5,-2.5) {$VO_3$};
\draw (0) to [out=325,in=215] (2);
\draw (1) -- (2);
\draw (2) -- (3);
\draw (0) to [out=35,in=145] (3);
\draw (1) to [out=35,in=145] (3);
\draw (0) to [out=325,in=215] (7);
\draw (1) to [out=325,in=215] (7);
\draw (2) to [out=325,in=215] (7);
\draw (3) to [out=325,in=215] (7);
\draw (4) to [out=325,in=215] (7);
\draw (5) to [out=325,in=215] (7);
\draw (6) -- (7);
\draw (0) to [out=35,in=145] (8);
\draw (1) to [out=35,in=145] (8);
\draw (2) to [out=35,in=145] (8);
\draw (3) to [out=35,in=145] (8);
\draw (4) to [out=35,in=145] (8);
\draw (5) to [out=35,in=145] (8);
\draw (6) to [out=35,in=145] (8);
\draw (7) -- (8);
\draw (0) to [out=325,in=215] (10);
\draw (1) to [out=325,in=215] (10);
\draw (2) to [out=325,in=215] (10);
\draw (3) to [out=325,in=215] (10);
\draw (4) to [out=325,in=215] (10);
\draw (5) to [out=325,in=215] (10);
\draw (6) to [out=325,in=215] (10);
\draw (7) to [out=325,in=215] (10);
\draw (8) to [out=325,in=215] (10);
\draw (9) -- (10);
\draw [thick, loosely dotted] (10) -- (11);
\draw [thick, decoration={brace,mirror,raise=5pt},decorate] (0b) -- (1b);
\draw [thick, decoration={brace,mirror,raise=5pt},decorate] (2b) -- (3b);
\draw [thick, decoration={brace,mirror,raise=5pt},decorate] (4b) -- (6b);
\draw [thick, decoration={brace,mirror,raise=5pt},decorate] (7b) -- (8b);
\draw [thick, decoration={brace,mirror,raise=5pt},decorate] (9b) -- (91b);
\draw [thick, decoration={brace,mirror,raise=5pt},decorate] (10b) -- (101b);
\draw [thick, decoration={brace,mirror,raise=5pt},decorate] (-0b) -- (-1b);
\draw [thick, decoration={brace,mirror,raise=5pt},decorate] (-2b) -- (-3b);
\draw [thick, decoration={brace,mirror,raise=5pt},decorate] (-4b) -- (-6b);
\draw [thick, decoration={brace,mirror,raise=5pt},decorate] (-7b) -- (-8b);
\draw [thick, decoration={brace,mirror,raise=5pt},decorate] (-9b) -- (-91b);
\draw [thick, decoration={brace,mirror,raise=5pt},decorate] (-10b) -- (-101b);
\end{tikzpicture}
\caption{$\G(P)$, with $P = (1,1,0,0,1,1,1,0,0,1,0,\ldots)$}\label{diagramgraph}
\end{figure}
\end{center}
\vspace{-0.8cm}

This gives us a framework in which to show that there are countably many non-isomorphic MB-homogeneous graphs. Recall (from \cite{hodges1993model}, say) that isomorphic $\sigma$-structures have the same \emph{definable sets}; a set of $n$-tuples $A$ is definable (without parameters) if there exists a $\sigma$-formula $\phi$ such that for all $x_i\in \mc{M}$, $\bar{x} = (x_1,\ldots,x_n)\in A$ if and only if $\phi(\bar{x})$ is true.

\begin{lemma}\label{ctblymb}
There are countably many pairwise non-isomorphic bi-equivalent MB-homogeneous graphs.
\end{lemma}

\begin{proof}
Define a binary sequence $P^n = (p_i)_{i\in\mb{N}}$ by the following:
\begin{equation*}
p_i = \begin{cases}1\text{ if } i = 0,1,2,\ldots, n-1, n+1, n+3,\ldots\\
0\text{ if }i = n,n+2,\ldots
\end{cases}
\end{equation*}
Here, $P^n$ is a sequence of $n$ many 1's followed by alternating 0's and 1's. In $\G(P^n)$, for all $k > n$ and all $i < k$, it follows that $v_i \sim v_k$ if and only if $v_i \nsim v_{k+1}$. Define the formula 
\[
\phi(x) = 
(\exists y \in V\G(P^n))
(\lnot (y = x))
(\forall z\in V\G(P^n))(x\sim z \Leftrightarrow y\sim z)
\] 
This formula identifies the set of vertices $v$ with the property that there is another vertex $ w \not\sim v$ such that $v$ and $w$ both have the same set of neighbours in the graph. 
From the definitions it is not hard to see that 
the set of all vertices $x$ that satisfy $\phi$ is $VI_1$. This is a definable set of size $n$. If $m \neq n$, the sets defined by $\phi(x)$ in $\G(P^m)$ and $\G(P^n)$ have different sizes, so the graphs cannot be isomorphic. From the remark following \autoref{gpmb}, it follows that $\G(P^m)$ and $\G(P^n)$ are bi-equivalent for any $m,n\in\mb{N}$.
\end{proof}

\begin{remark}
It can be shown (in a similar fashion to the proof to \autoref{mbfrucht} later in the section) that \[\text{Aut}(\G(P)) = \prod_{i\in\mb{N}}(\text{Aut}(\G(O_k))\times \text{Aut}(\G(I_k))).\] It follows that Aut$(\G(P^n))$ is the infinite direct product of one copy of Sym$(n)$ together with infinitely many trivial groups; so Aut$(\G(P^n))\cong$ Sym$(n)$; providing an alternate proof to \autoref{ctblymb}.
Furthermore, it is straightforward to see that if two monoids $T,T'$ have non-isomorphic groups of units $U,U'$ respectively, then $T\ncong T'$. Hence we have constructed examples of non-isomorphic oligomorphic permutation monoids occurring as bimorphism monoids of bi-equivalent relational structures. By \autoref{xyorbits}, this means that each of these oligomorphic permutation monoids have the same strong orbits. 
\end{remark}

We now aim to use the framework established here to construct $2^{\aleph_0}$ non-isomorphic examples of MB-homogeneous graphs. The idea is to add in pairwise non-embeddable finite graphs into the age of some $\G(P)$ to ensure uniqueness up to isomorphism. To this end, let $A = (a_n)_{n\in\mb{N}}$ be a strictly increasing sequence of natural numbers. We use $A$ to recursively define a binary sequence $PA = (p_i)_{i\in\mb{N}}$ as follows:
\begin{description}
\item[Base:] 0 followed by $a_1$ many 1's.
\item[Inductive:] Assuming that the $n^{\text{th}}$ stage of the sequence has been constructed, add a 0 followed by $a_{n+1}$ many 1's to the right hand side of the sequence.
\end{description}
For instance, if $A = (2,3,5,\ldots)$ then $PA = (0,1,1,0,1,1,1,0,1,1,1,1,1,\ldots)$. Using such a binary sequence $PA$, we construct $\G(PA)$ in the fashion of \autoref{urexample}. As $PA$ has infinitely many of each term then $\G(PA)$ is MB-homogeneous by \autoref{gpmb}. The eventual plan is to induce finite graphs onto the independent sets induced on $\G(PA)$ by strings of consecutive 1's in $PA$. By selecting a suitable countable family of pairwise non-embeddable graphs that do not appear in the age of $\G(PA)$, we ensure graphs of different ages. 

We prove two lemmas which form the basis for our construction of $2^{\aleph_0}$ non-isomorphic examples of MB-homogeneous graphs. The second lemma is folklore; and the proof is omitted.

\begin{lemma} Let $P$ be any binary sequence. Then $\G(P)$ does not embed any cycle graph of size $m\geq 4$.
\end{lemma}

\begin{proof} Let $VX = \{v_{i_1},\ldots,v_{i_m}\}\subseteq V\G(P)$, with $m\geq 4$. Let $\G(X)$ be the graph on $VX$ with edges induced by $\G(P)$. The degree of $v_{i_m}$ in $\G(X)$ is either $0$ or $m - 1$; in particular, it is not $2$. So $\G(X)$ is not an $m$-cycle.
\end{proof}

\begin{lemma}\label{cycles} Let $C_m$ and $C_n$ be two cycle graphs with $m,n\geq 3$. Then $C_m$ embeds in $C_n$ (and vice versa) if and only if $m=n$; in which case they are isomorphic. \dne
\end{lemma}

Hence $\ms{C} = \{C_n$ : $n\geq 4\}$ is a countable family of pairwise non-embeddable graphs that do not appear as induced subgraphs of any $\G(P)$. Now suppose that $A = (a_n)_{n\in\mb{N}}$ is a strictly increasing sequence of natural numbers with $a_i\geq 4$. Construct $\G(PA)$ as outlined above. We have that the size of each $I_k$ in $PA$ is $a_k$. For each independent set $\G(I_K)\subseteq \G(PA)$, draw an $a_k$-cycle on its vertices, thus creating a new graph $\G(PA)'$ (see \autoref{cyclepic}). 

\begin{center}
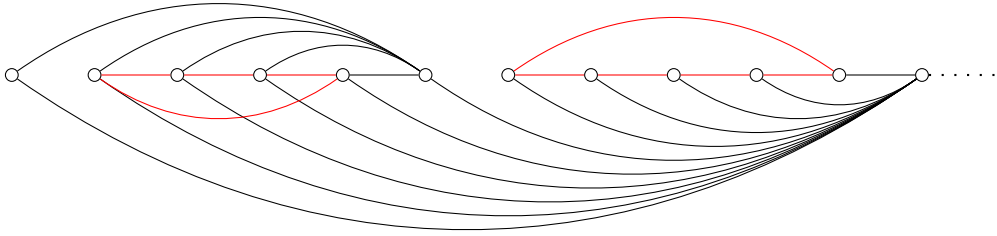
\begin{figure}[h]
\centering
\begin{tikzpicture}[node distance=2cm,inner sep=0.6mm,scale=1.1]
\node (0) at (-6,0) [circle,draw] {};
\node (1) at (-5,0) [circle,draw] {};
\node (2) at (-4,0) [circle,draw] {};
\node (3) at (-3,0) [circle,draw] {};
\node (4) at (-2,0) [circle,draw] {};
\node (5) at (-1,0) [circle,draw] {};
\node (6) at (0,0) [circle,draw] {};
\node (7) at (1,0) [circle,draw] {};
\node (8) at (2,0) [circle,draw] {};
\node (9) at (3,0) [circle,draw] {};
\node (10) at (4,0) [circle,draw] {};
\node (11) at (5,0) [circle,draw] {};
\node (12) at (6,0) {};
\draw (0) to [out=35,in=145] (5);
\draw (1) to [out=35,in=145] (5);
\draw (2) to [out=35,in=145] (5);
\draw (3) to [out=35,in=145] (5);
\draw (4) -- (5);
\draw (0) to [out=325,in=215] (11);
\draw (1) to [out=325,in=215] (11);
\draw (2) to [out=325,in=215] (11);
\draw (3) to [out=325,in=215] (11);
\draw (4) to [out=325,in=215] (11);
\draw (5) to [out=325,in=215] (11);
\draw (6) to [out=325,in=215] (11);
\draw (7) to [out=325,in=215] (11);
\draw (8) to [out=325,in=215] (11);
\draw (9) to [out=325,in=215] (11);
\draw (10) -- (11);
\draw [thick, loosely dotted] (11) -- (12);
\draw [red] (1) -- (2);
\draw [red] (2) -- (3);
\draw [red] (3) -- (4);
\draw [red] (1) to [out=325,in=215] (4);
\draw [red] (6) -- (7);
\draw [red] (7) -- (8);
\draw [red] (8) -- (9);
\draw [red] (9) -- (10);
\draw [red] (6) to [out=35,in=145] (10);
\end{tikzpicture}
\caption{$\G(PA)'$ corresponding to the sequence $A = (4,5,6,\ldots)$, with added cycles highlighted in red}\label{cyclepic}
\end{figure}
\end{center} 
\vspace{-0.8cm}

Note that even with these additional structures, $\G(PA)$ is still MB-homogeneous as it has properties $\tr$ and $\tf$. Since we have added so many structures to the age, we must be careful that we have not added extra cycles of sizes not expressed in the sequence $A$. The next proposition alleviates this concern.

\begin{proposition}\label{nomcycle} Suppose that $A = (a_n)_{n\in\mb{N}}$ is a strictly increasing sequence of natural numbers with $a_1\geq 4$ and suppose that $m\geq 4$ is a natural number such that $m\neq a_n$ for all $n\in\mb{N}$. Then the graph $\G(PA)'$ outlined above does not contain an $m$-cycle.
\end{proposition}

\begin{proof} Suppose that $M\subseteq \G(PA)'$ is an $m$-cycle. Then the edge set of $M$ is a combination of the edges of the graph induced by the finite subsequence $Q = (q_{i_1},\ldots,q_{i_m})$ of $PA$ (where $i_1<i_2<\ldots<i_m$) on $VM = \{v_{i_1},\ldots,v_{i_m}\}$ and the edges from cycles added onto $\G(PA)$. We aim to show that $M = \G(I_n)$ for some $n\in\mb{N}$, that is, the only cycles of size $\geq 4$ in $\G(PA)'$ are precisely those we added in the construction. So assume here that $Q$ contains a 0. As $M$ is an $m$-cycle, $d_M(v)=2$ for all $v\in M$ and so the only elements of $Q$ that can be 0 are $q_{i_1},q_{i_2}$ and $q_{i_3}$. We split our consideration into cases.

\begin{case1}[$q_{i_3} = 0$] As we assume this, $v_{i_3}$ is adjacent to both $v_{i_1}$ and $v_{i_2}$. If $q_{i_2} = 0$, then $v_{i_2} \sim v_{i_1}$ creating a 3-cycle; this is a contradiction as $M$ does not embed a 3-cycle. Therefore, $q_{i_2} = 1$ and so $v_{i_1}$ is adjacent to some $v_{i_j}$ where $3<j\leq m$. Since $q_{i_j} = 1$ for all $i_j > i_3$, it follows that the edge between $v_{i_1}$ and $v_{i_j}$ was induced by an added cycle. This implies that $v_{i_1},v_{i_j}\in V\G(I_k)$ for some $k$. Therefore, $(b_{i_1},\ldots,b_{i_j})$ is a sequence of 1's where $i_1<i_2<\ldots<i_j$ are consecutive natural numbers; a contradiction as $i_1<i_3<i_j$ and so $1 = q_{i_3} = 0$. Hence, $q_{i_3} \neq 0$.
\end{case1}

\begin{case2}[$q_{i_2} = 0$] Since this happens, $v_{i_2}$ is adjacent to $v_{i_1}$ and some vertex $v_{i_j}\neq v_{i_1}$; so $i_j > i_2$. If $q_{i_j}$ is 1, then $v_{i_j}$ is non-adjacent to any vertex $v_{i_k}$ with $i_k<i_j$ and $q_{i_k} = 0$, as the construction of $\G(PA)'$ ensures that no edges are drawn in this case. But this is a contradiction as $q_{i_2} = 0$ and $v_{i_j}\sim v_{i_2}$. So $q_{i_j}$ must be 0 in this case; but $i_j >i_2$ and so $q_{i_j} = 1$ by Case 1 and the argument preceding Case 1. This is a contradiction and so $q_{i_2} = 1$.
\end{case2} 

A similar argument to that of Case 2 holds for when $q_{i_1}=0$ and so no element of $Q$ is 0. Hence, $Q$ is made up of 1's; however, it is still a possibility that these originate from from different $I_k$'s. We now show however that every vertex of $M$ comes from a single $VI_k$ for some $k$. As no element of $Q$ is 0, we conclude that two vertices in $M$ have an edge between them only if they are contained in the same $VI_k$ and have an edge between them in $\G(I_k)$. As $M$ is connected, we have that $M\subseteq \G(I_k)$ for some $k$. Finally, as $\G(I_k)$ is an $a_k$-cycle embedding an $m$-cycle, we are forced to conclude that $m=a_k$ by \autoref{cycles} and so we are done.
\end{proof}

\begin{corollary}\label{noniso} Suppose that $A=(a_n)_{n\in\mb{N}}$ and $B=(b_n)_{n\in\mb{N}}$ are two different strictly increasing sequences of natural numbers with $a_1,b_1 \geq 4$. Then $\G(PA)'\ncong \G(PB)'$.
\end{corollary}

\begin{proof} As $A$ and $B$ are different sequences there exists a $j\in\mb{N}$ such that $a_j\neq b_j$; without loss of generality assume that $a_j<b_j$. Hence $\G(PA)'$ embeds an $a_j$-cycle; but as $a_j\notin B$, by \autoref{nomcycle} $\G(PB)'$ does not embed an $a_j$-cycle. Hence Age$(\G(PA)')\neq$ Age$(\G(PB)')$ and so they are not isomorphic.
\end{proof}

These results prove the following:

\begin{theorem}\label{uncmany} There are $2^{\aleph_0}$ many countably infinite, non-isomorphic, non-bi-equivalent, MB-homogeneous graphs, each of which is equivalent to the random graph $R$ up to bimorphisms.
\end{theorem}

\begin{proof} As there are $2^{\aleph_0}$ strictly increasing sequences of natural numbers we have continuum many examples of $\G(PA)'$ by \autoref{noniso}. As these graphs have different ages, this means we have constructed $2^{\aleph_0}$ many countably infinite, non-isomorphic, non-bi-equivalent MB-homogeneous graphs. Furthermore, as each these examples has property $\tr$ and $\tf$, they are equivalent to $R$ up to bimorphisms by \autoref{beqr}.
\end{proof}

\begin{remark} As a consequence of this, $|[R]^{\sim_b}| = 2^{\aleph_0}$. This means that there are $2^{\aleph_0}$ pairwise non-isomorphic graphs $\G$ with the property stated in the remark after \autoref{beqr}.\end{remark}

Finally in this section, we utilise this technique of overlaying finite graphs in order to prove the second main theorem of the section. Before we do this, here is an important lemma. Recall that $N(v)$ denotes the neighbourhood set of a vertex $v\in V\G$.

\begin{lemma}\label{neighbourhoodi}\label{neighbourhood} Let $P=(p_n)_{n\in\mb{N}_0}$ be a binary sequence, and let $\G(P)$ be the graph determined by this binary sequence. 
\begin{enumerate}[(1)]
 \item If $v_j$ is a vertex in $VI_n$ for some $n\in\mb{N}$, then $\G(N(v_j))\cong K^{\aleph_0}$.
 \item Suppose that $v,w$ are vertices in some $VO_k = \{v_{k_1},\ldots,v_{k_n}\}$. Then $N(v)\cup \{v\} = N(w)\cup \{w\}$ and $\G(N(v))\cong \G(N(w))$.
\end{enumerate}
\end{lemma}

\begin{proof} 
\begin{enumerate}[(1)]
 \item From the assumption that $p_j = 1$, the observation of \autoref{urexample}, and the definition of an edge in $\G(P)$, we have that $v_j\sim v_k$ if and only if $j<k$ and $p_k = 0$; so $N(v_j) = \{v_k$ : $k>j,\; p_k = 0\}$. As there are infinitely many 0's in $P$, this set is infinite as $j$ is finite. Now, take $v_a,v_b\in N(v_j)$. Here, $v_a\sim v_b\in \G(N(v_j))$ if and only if $p_{\text{max}(a,b)} = 0$; but $p_a = p_b = 0$ and so any two vertices of $N(v_j)$ are adjacent.
 
 \item Define the sets $X = \{v_j$ : $j<k_1\}$ and $Y = \{v_m$ : $m\geq k_1,\; p_m = 0\}$. Due to the construction of $\G$ in \autoref{urexample}, we have that $N(v)\cup \{v\} = X\cup Y = N(w)\cup \{w\}$. For all $u\in V\G(VO_k)$, it is easy to see that $N(u) = X\cup (Y\smallsetminus \{u\})$. Using a similar argument to the proof of (1), we have that $\G(Y\smallsetminus\{u\})$ is an infinite complete graph for any $u\in \G(VO_k)$, and every element in $Y$ is adjacent to every element of $X$. For any $v,w\in VO_k$ define a map $f:X\cup (Y\smallsetminus\{v\})\rarr{}X\cup (Y\smallsetminus\{w\})$ fixing $X$ pointwise and sending $(Y\smallsetminus\{v\})$ to $(Y\smallsetminus\{w\})$ in any fashion; this is an isomorphism between $\G(N(v))$ and $\G(N(w))$.
\end{enumerate}
\end{proof}

Before our next proposition, recall that if there exists $\gamma\in$ Aut$(\G)$ such that $v\gamma = w$, then the graphs induced on $N(v)$ and $N(w)$ are isomorphic. Furthermore, recall that an independent set $U\subseteq \G$ is a \emph{maximum independent set} if there does not exist a set $W\supsetneq U$ such that $W$ is independent. Finally, a graph $\G$ is \emph{$n$-regular} if every $v\in V\G$ has degree $n$ for some $n\in\mb{N}$.

\begin{theorem}\label{mbfrucht} Any finite group $H$ arises as the automorphism group of an MB-homogeneous graph $\G$.
\end{theorem}

\begin{proof}
By a version of Frucht's theorem \cite{frucht1949graphs} there exists countably many 3-regular graphs $\G$ such that Aut$(\G) \cong H$; and so, as there are only finitely many graphs of size less than or equal to 5, there exists a 3-regular graph $\Delta$ of size $n\geq 6$ such that Aut$(\Delta) \cong H$. By the handshake lemma (see p5 \cite{diestel2000graph}), such a graph must have a total of $3n/2$ edges out of a total of $(n^2-n)/2$ possible edges; as $n\geq 6$, this means that $\Delta$ must induce at least 6 non-edges.

Define a binary sequence $P = (p_i)_{i\in\mb{N}}$ by the following:
\begin{equation*}
p_i = \begin{cases}1\text{ if } i = 0,1,2,\ldots, n-1, n+1, n+3,\ldots\\
0\text{ if }i = n,n+2,\ldots
\end{cases}
\end{equation*}
So $P$ is a sequence of $n$ many 1's followed by alternating 0's and 1's. Using the notation established above in \autoref{io}, it follows that $|I_1| = n$ and $|O_k| = |I_m| = 1$ for $k\geq 1$ and $m\geq 2$. Construct $\G(P)$ on $V\G(P) = \{v_0,v_1,\ldots\}$ as illustrated in \autoref{urexample}, and draw in edges on $VI_1$ such that $\G(VI_1)\cong\Delta$ to obtain a graph $\G(P)'$ (see \autoref{gpprime}).

\begin{center}
\begin{figure}[h]
\centering
\begin{tikzpicture}[node distance=2cm,inner sep=0.7mm,scale=1.35]
\node (0) at (-4,0) [circle,draw,red,label=below:$v_0$] {};
\node (1) at (-3,0) [circle,draw,red,label=below:$v_1$] {};
\node (11) at (-3.25,0.25) {};
\node (111) at (-3.25,0) {};
\node (12) at (-3.25,-0.25) {};
\node (2) at (-2,0) [circle,draw,red,label=below:$v_2$] {};
\node (3) at (-1,0) [circle,draw,red,label=below:$v_{n-2}$] {};
\node (4) at (0,0) [circle,draw,red,label=below:$v_{n-1}$] {};
\node (41) at (-0.75,0.25) {};
\node (411) at (-0.75,0) {};
\node (42) at (-0.75,-0.25) {};
\node (5) at (1,0) [circle,draw,label=below:$v_n$] {};
\node (6) at (2,0) [circle,draw,label=below:$v_{n+1}$] {};
\node (7) at (3,0) [circle,draw,label=below:$v_{n+2}$] {};
\node (8) at (4,0) [circle,draw,label=below:$v_{n+3}$] {};
\node (9) at (5,0) [circle,draw,label=below:$v_{n+4}$] {};
\node (10) at (6,0) {};

\draw (0) to [out=35,in=145] (5);
\draw (1) to [out=35,in=145] (5);
\draw (2) to [out=35,in=145] (5);
\draw (3) to [out=35,in=145] (5);
\draw (4) -- (5);
\draw (0) to [out=325,in=215] (7);
\draw (1) to [out=325,in=215] (7);
\draw (2) to [out=325,in=215] (7);
\draw (3) to [out=325,in=215] (7);
\draw (4) to [out=325,in=215] (7);
\draw (5) to [out=325,in=215] (7);
\draw (6) -- (7);
\draw (0) to [out=35,in=145] (9);
\draw (1) to [out=35,in=145] (9);
\draw (2) to [out=35,in=145] (9);
\draw (3) to [out=35,in=145] (9);
\draw (4) to [out=35,in=145] (9);
\draw (5) to [out=35,in=145] (9);
\draw (6) to [out=35,in=145] (9);
\draw (7) to [out=35,in=145] (9);
\draw (8) -- (9);
\draw [thick, loosely dotted] (2) -- (3);
\draw [thick, loosely dotted] (9) -- (10);
\path [draw half paths={solid,red}{dotted,red}]  (0) -- (11);
\path [draw half paths={solid,red}{dotted,red}]  (0) -- (111);
\path [draw half paths={solid,red}{dotted,red}]  (0) -- (12);
\path [draw half paths={solid,red}{dotted,red}]  (4) -- (41);
\path [draw half paths={solid,red}{dotted,red}]  (4) -- (411);
\path [draw half paths={solid,red}{dotted,red}]  (4) -- (42);
\end{tikzpicture}
\caption{$\G(P)'$, with $\Delta$ highlighted in red}\label{gpprime}
\end{figure}
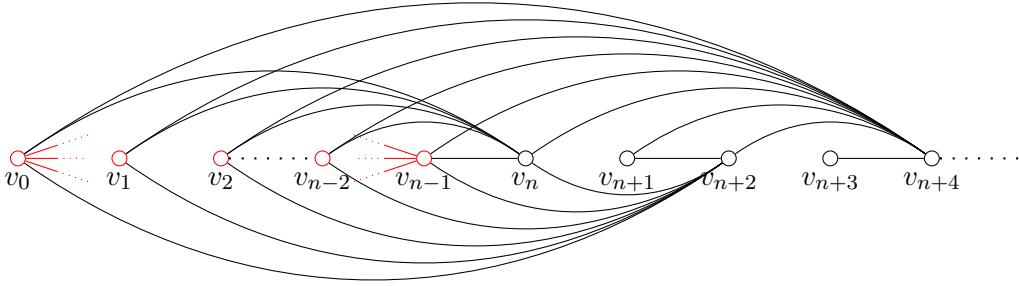
\end{center}
\vspace{-0.8cm}

We aim to show that $VO_i$ and $VI_i$ are fixed setwise for any automorphism of $\G(P)'$ through a series of claims.

\begin{claim1} If $v_a\in VO_i$ and $v_b\in VO_j$ with $i \neq j$, then $\G(N(v)) \ncong \G(N(w))$.
\end{claim1}

\noindent\emph{Proof of Claim 1.} As $|VO_k|=1$ for all $k\in\mb{N}$, we have that $VO_i = \{v_a\}$ and $VO_j = \{v_b\}$. Assume without loss of generality that $a<b$. We define the following sets:
\begin{eqnarray*}
X_a &=& \{v_k\text{ : }k<a\}\\
X_b &=& \{v_k\text{ : }k<b\}\\
Y_a &=& \{v_k\text{ : }k\geq a,\; p_k = 0\}\\
Y_b &=& \{v_k\text{ : }k\geq b,\; p_k = 0\}.
\end{eqnarray*}

\autoref{neighbourhood} (2) applies in this situation; so $N(v_a) = X_a\cup Y_a$ and $N(v_b)= X_b\cup Y_b$. Since $\G(Y_a)$ and $\G(Y_b)$ are infinite complete graphs, it follows that any maximum independent set of $N(v_a),N(v_b)$ is contained in $X_a,X_b$ respectively. Now, as $\G(VI_1)\cong \Delta$ is a 3-regular graph on more than six vertices, there exists some maximal independent set $M\subseteq VI_1$ of $\G(VI_1)$ with size greater than or equal to 2.

Now we consider the sets \[A = \{v_c\in V\G(P)'\text{ : }n-1<c<a,\; p_c = 1\}\cup M\] and \[B = \{v_c\in V\G(P)'\text{ : }n-1<c<b,\; p_c = 1\}\cup M,\] the maximum independent sets of $X_a$ and $X_b$ respectively. As $i<j$, there exists $d$ such that $a<d<b$ with $p_d=1$. Hence $v_d\in B\smallsetminus A$ and so $|B|>|A|$. Since $A,B$ are maximum independent sets of $\G(N(v_a))$ and $\G(N(v_b))$ respectively with different sizes, we conclude that $\G(N(v_a))\ncong \G(N(v_b))$. This ends the proof of Claim 1.\\

\noindent This shows that there exists no automorphism $\gamma$ of $\G(P)'$ sending any $v\in VO_i$ to $w\in VO_j$ with $i\neq j$.

\begin{claim2} There exists no automorphism sending $v\in VO_k$ to $w\in VI_m$ for all $k,m\in\mb{N}$.
\end{claim2}

\noindent\emph{Proof of Claim 2}. We split the proof into two cases; where $m=1$ and where $m\geq 2$. For the latter, $\G(N(w))\cong K^{\aleph_0}$ for any $w\in VI_m$ with $m\geq 2$ by \autoref{neighbourhoodi} (1). But as $\G(VI_1)$ is not a complete graph, we have that $\G(VO_k)$ contains a non-edge for all $k\in\mb{N}$ and so $\G(N(v))\ncong\G(N(w))$ in this case. It remains to show that there is no automorphism sending $v\in VO_k$ to $w\in VI_1$. In this case $\G(N(w))$ is the union of an infinite complete graph $K$ and $G =\G(N_{\G(VI_1)}(w))$, with every vertex of $K$ connected to every vertex of $G$. This means any non-edge of $\G(N(w))$ must be induced by $G$; as $|N_{\G(VI_1)}(w)|=3$, there are at most 3 of them for any $w\in VI_1$. However, as $X_a$ contains $VI_1$, we have that $\G(N(v))$ contains $\Delta$ as an induced subgraph. By the reasoning above, $\Delta$ has at least 6 non-edges and therefore so does $\G(N(v))$. This means that $\G(N(v))\ncong \G(N(w))$ for any $v\in VO_k$ and any $w\in VI_1$ and so we are done.\\

\noindent Here, $VO_k$ is fixed setwise for all $k\in\mb{N}$; as $|VO_k|=1$ for all such $k$ we have that they are also fixed pointwise. We have that any $VI_{k}$ sandwiched between $VO_{k-1}$ and $VO_{k}$ are the only vertices not adjacent to every vertex in $VO_k$ and adjacent to every vertex in $VO_{k+1}$. As $VO_k$ and $VO_{k+1}$ are fixed setwise, we deduce that $VI_k$ is fixed setwise (and hence pointwise) for $k\geq 2$. We conclude that $VI_1$ is fixed setwise under automorphisms of $\G(P)'$.

Finally, we show that any bijective map $\gamma:\G(P)'\to\G(P)'$ acting as an automorphism on $VI_1$ and fixing everything else is an automorphism of $\G(P)'$. As every $v\in VO_k$ for all $k$ is connected to each $u\in VI_1$ and every $w\in VI_m$ for all $m$ is independent of each $u\in VI_1$, we have that this map preserves all edges and non-edges of $\G(P)'$ and so is an automorphism of $\G(P)'$.
\end{proof}

Using this together with \autoref{xyotm}, it follows that for any finite group $U$ there exists an oligomorphic permutation monoid that has $U$ as a group of units. 

\section{Worked examples}\label{s4}

This section is devoted to determining MB-homogeneity of examples of homogeneous structures. This investigation concludes in a complete classification of countable homogeneous graphs that are also MB-homogeneous.

There are a variety of ways to demonstrate that a structure $\mc{M}$ is MB-homogeneous. Firstly, if $\mc{M}$ is a homogeneous structure where each finite partial monomorphism $h$ is also a finite partial isomorphism, then $\mc{M}$ is MB-homogeneous as we can extend $h$ to an automorphism. If $\mc{M}$ is a graph, then it suffices to show that $\mc{M}$ has properties $\tr$ and $\tf$ by \autoref{trtfmb}. Finally, recall the remarks of \autoref{mepbep}; to prove that $\mc{M}$ is MB-homogeneous it suffices to show that $\mc{M}$ has both the 1PAMEP and the 1PMEP. All of these techniques are used at points in demonstrating the following positive examples. 

\begin{example}\label{kcgraph} Let $\Gamma = K^{\aleph_0}$, the complete graph on countably many vertices; this graph is homogeneous \cite{lachlan1980countable}. The proof that $\G$ is MB-homogeneous is equivalent to saying that any injective map between finite subsets of a countable set can be extended to a permutation of the set. Another proof relies on homogeneity, and is more open to generalisation. Here, suppose that $h:A\rarr{}B$ is a monomorphism between two finite substructures of $\Gamma$. As there are no non-edges to preserve, it must preserve non-edges and so $h$ is a finite partial isomorphism. Using homogeneity of $\G$, we extend $h$ to an automorphism (and hence a bimorphism) of $\G$ and so $\G$ is MB-homogeneous. Its complement $\bar{\Gamma}$, the infinite null graph, is also MB-homogeneous by \autoref{complement}.
\end{example}

\begin{example}\label{tournament} A tournament is defined to be an oriented, loopless complete graph. By a similar argument to the complete graph in \autoref{kcgraph}, every finite partial monomorphism of a tournament is a finite partial isomorphism. So the three homogeneous tournaments $(\mb{Q},<)$, the random tournament $T$ and the local order $S(2)$ (see Cherlin \cite{cherlin1998classification}) are all MB-homogeneous.
\end{example}

\begin{example}\label{random} Let $R$ be the countable universal homogeneous undirected graph, also known as the \emph{random graph}. It follows from the extension property characteristic to the random graph (see \autoref{epinr}) that $R$ has properties $\tr$ and $\tf$ (see \autoref{deftrtf}) and so is an MB-homogeneous graph by \autoref{trtfmb}.
\end{example}

\begin{example}\label{generic} Let $D$ be the Fra\"{i}ss\'{e} limit of the class of directed graphs without loops or 2-cycles, otherwise known as the \emph{generic digraph}. It is well known that $D$ satisfies the following extension property (see \cite{agarwal2016reducts}).

\begin{quote} (DEP) For any finite and pairwise disjoint sets of vertices $U,V,W$ of $D$, there exists a vertex $x$ of $D$ such that there is an arc from $x$ to every element of $U$, an arc to $x$ from every element of $V$, and $x$ is independent of every vertex in $W$. (see \autoref{preoarpdi} for a diagram of an example.)
\end{quote}
\begin{center}
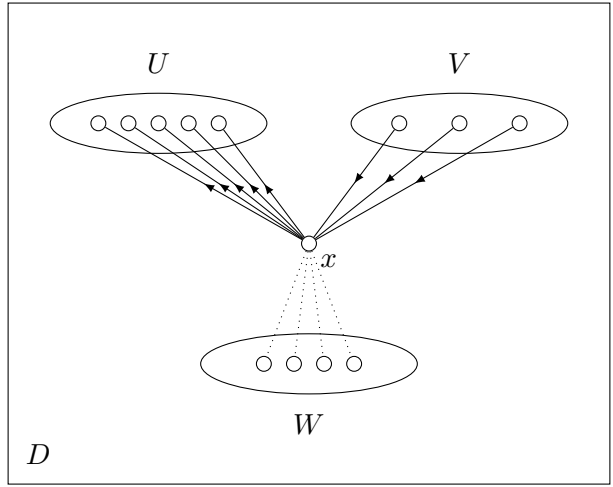
\begin{figure}[h]
\centering
\begin{tikzpicture}[node distance=1.8cm,inner sep=0.7mm,scale=0.8]
\node(R) at (-4.5,-5.5) {$D$};
\node(C) at (0,0) {};
\node(U5) at (-1.5,0) [circle,draw] {};
\node(U4) at (-2,0) [circle,draw] {};
\node(U3) at (-2.5,0) [circle,draw] {};
\node(U2) at (-3,0) [circle,draw] {};
\node(U1) at (-3.5,0) [circle,draw] {};
\node(V1) at (1.5,0) [circle,draw] {};
\node(V2) at (2.5,0) [circle,draw] {};
\node(V3) at (3.5,0) [circle,draw] {};
\node(W1) at (-0.75,-4) [circle,draw] {};
\node(W2) at (-0.25,-4) [circle,draw] {};
\node(W3) at (0.25,-4) [circle,draw] {};
\node(W4) at (0.75,-4) [circle,draw] {};
\node(Wx) at (0,-4) {};
\node(W) at (0,-5) {$W$};
\node(U) at (-2.5,1) {$U$};
\node(V) at (2.5,1) {$V$};
\node(x) at (0,-2) [circle,draw,label=below right:$x$] {};
\draw(U3) ellipse (1.8cm and 0.5cm);
\draw(V2) ellipse (1.8cm and 0.5cm);
\draw(Wx) ellipse (1.8cm and 0.5cm);
\draw (-5,-6) rectangle (5,2);
\draw[middlearrow={latex}] (x) -- (U5);
\draw[middlearrow={latex}] (x) -- (U4);
\draw[middlearrow={latex}] (x) -- (U3);
\draw[middlearrow={latex}] (x) -- (U2);
\draw[middlearrow={latex}] (x) -- (U1);
\draw[middlearrow={latex}] (V1) -- (x);
\draw[middlearrow={latex}] (V2) -- (x);
\draw[middlearrow={latex}] (V3) -- (x);
\draw[dotted] (x) -- (W1);
\draw[dotted] (x) -- (W2);
\draw[dotted] (x) -- (W3);
\draw[dotted] (x) -- (W4);
\end{tikzpicture}
\caption{Example of directed extension property in $D$}\label{preoarpdi}
\end{figure}
\end{center}
\vspace{-0.8cm}

Using this, we show that $D$ is MB-homogeneous by demonstrating that it has the 1PMEP and 1PAMEP in turn. 

Suppose that $A\subseteq B\in$ Age$(D)$ with $B\smallsetminus A=\{b\}$ and that $f:A\rarr{}D$ is a monomorphism. Decompose $A$ into three disjoint sets $b^{\rarr{}} =\{a\in A$ : $b\rarr{}a\}$, $b^{\larr{}} = \{a\in A$ : $b\larr{}a\}$ and $b^\parallel = \{a\in A$ : $b\parallel a\}$. The fact that $f$ is injective means that the sets $b^{\rarr{}}f,b^{\larr{}}f$ and $b^\parallel f$ are pairwise disjoint subsets of $VD$. Using the DEP, select a vertex $x\in VD$ such that $x$ has an arc to all elements of $b^{\rarr{}}f$, an arc from all elements of $b^{\larr{}}f$ and is independent of all elements of $b\parallel f$. Define $g:B\rarr{}D$ to be the map such that $bg = x$ and $g|_A = f$; due to our choice of $x$, this is a monomorphism. So $D$ has the 1PMEP. 

Now suppose that $X\subseteq Y\in$ Age$(D)$ with $Y\smallsetminus X = \{y\}$ and that $\bar{f}:X\to D$ is an antimonomorphism. The fact that $X$ is finite implies that im $\bar{f}$ is finite, and so there exists a vertex $w\in VD$ such that $w$ is independent of all elements in $X\bar{f}$ by the DEP. Define $\bar{g}:Y\to D$ to be the function such that $y\bar{g} = w$ and $\bar{g}|_X = \bar{f}$; thanks to our choice of $w$, this is an antimonomorphism as all non-arcs are preserved. So $D$ has the 1PAMEP and hence $D$ is MB-homogeneous by \autoref{mepbep}.
\end{example}

\begin{remark} It can be shown that the class of finite loopless digraphs with 2-cycles is a Fra\"{i}ss\'{e} class; let $D^*$ be the Fra\"{i}ss\'{e} limit of this class. Here, $D^*$ has a slightly different extension property (see Ch. 4 of \cite{mcphee2012endomorphisms}). With this, we can use a similar proof to above to show that $D^*$ is MB-homogeneous.
\end{remark}

We continue our investigation by focusing on a class of structures known as \emph{cores}; a structure $\mc{M}$ is a core if every endomorphism of $\mc{M}$ is an embedding. Every $\aleph_0$-categorical structure has a core, and is homomorphically equivalent to a model-complete core \cite{bodirsky2005core}. Cores play an important role in the theory of constraint satisfaction problems; see Bodirsky's habilitation thesis \cite{bodirsky2012complexity} for an introduction to the topic. Widely studied examples of cores include the countable dense linear order without endpoints $(\mb{Q},<)$, the complete graph on countably many vertices $K^{\aleph_0}$, and its complement $\bar{K}^{\aleph_0}$. Proving that a structure $\mc{M}$ is a core is a useful way to show that a structure is \emph{not} MB-homogeneous, as this next result shows.

\begin{lemma}\label{123lem} 
Let $\mc{M}$ be a core such that there exists a finite partial monomorphism of $\mc{M}$ that is not an isomorphism. Then $\mc{M}$ is not MH-homogeneous (and hence, not MB-homogeneous).
\end{lemma}

\begin{proof} Let $h$ be a finite partial monomorphism of a core $\mc{M}$ that is not an isomorphism. As any endomorphism of $\mc{M}$ is an embedding, we cannot extend $h$.
\end{proof}

We can use this result to detail some homogeneous structures that are not MB-homogeneous.

\begin{example}\label{knfreegraphs} Let $G$ be the countable homogeneous $K_n$-free graph for $n\geq 3$. By results of Mudrinski in \cite{mudrinski2010notes}, it is shown that $G$ is a core for all such $n$. As there are finite partial monomorphisms of $G$ that are not isomorphisms (send any non-edge to an edge, for instance) we have that $G$ is not MB-homogeneous by \autoref{123lem}.
\end{example}

\begin{example}\label{mhhhendig}
The oriented graph analogues of the homogeneous $K_n$-free graphs are the \emph{Henson digraphs} $M_T$. These are the Fra\"{i}ss\'{e} limits of the class of all digraphs not embedding elements of some set $T$ of finite tournaments. We show that $M_T$ is a core for any non-empty $T$ containing tournaments of three or more vertices.

Suppose for a contradiction there exists a $\gamma\in$ End$(M_T)$ such that for some independent pair of vertices $v,w$ we have that $v\gamma = w\gamma$. Select a tournament $Y\in T$ with the least number of vertices, and choose $x,y\in Y$ such that $x\rarr{}y$. Create a oriented graph $Y'$ by removing the arc between $x$ and $y$, adding an extra vertex $x'$ and drawing an arc $x'\rarr{}y$; see \autoref{yfig} for an example.

\begin{figure}[h]
\centering
\begin{tikzpicture}[node distance=2cm,inner sep=0.7mm,scale=1.7]
\node (y) at (-1.5,1) {$Y$};
\node (-1a) at (-2,0.5) [circle,draw] {};
\node (1a) at (-1,0.5) [circle,draw] {};
\node (-1b) at (-2,-0.5) [circle,draw,label=below:$x$] {};
\node (1b) at (-1,-0.5) [circle,draw,label=below:$y$] {};
\draw[middlearrow={latex}] (-1a) -- (-1b);
\draw[middlearrow={latex}] (1a) -- (1b);
\draw[middlearrow={latex}] (-1b) -- (1b);
\draw[middlearrow={latex}] (-1a) -- (1a);
\draw[middlearrow={latex}] (1a) -- (-1b);
\draw[middlearrow={latex}] (1b) -- (-1a);
\node (yd) at (1.5,1) {$Y'$};
\node (2a) at (2,0.5) [circle,draw] {};
\node (-2a) at (1,0.5) [circle,draw] {};
\node (2b) at (2,-0.5) [circle,draw,label=below:$y$] {};
\node (-2b) at (1,-0.5) [circle,draw,label=below:$x$] {};
\node (2c) at (1.2,-1) [circle,draw,label=below:$x'$] {};
\draw[middlearrow={latex}] (2a) -- (2b);
\draw[middlearrow={latex}] (-2a) -- (-2b);
%\draw[middlearrow={latex}] (-2b) -- (2b);
\draw[middlearrow={latex}] (-2a) -- (2a);
\draw[middlearrow={latex}] (2a) -- (-2b);
\draw[middlearrow={latex}] (2b) -- (-2a);
\draw[middlearrow={latex}] (2c) -- (2b);
\end{tikzpicture}
\caption{Construction of $Y'$ from $Y$ in \autoref{mhhhendig}}\label{yfig}
\end{figure}
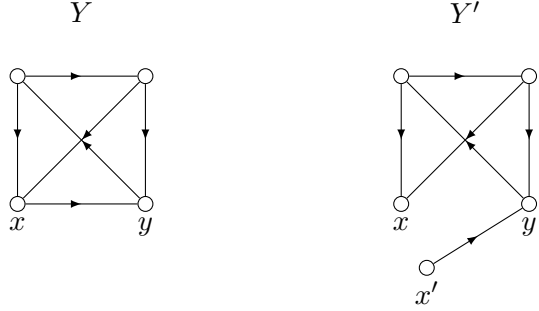

Note that there is no tournament on any $|Y|$-set of vertices of $Y'$; so $Y'\in$ Age$(M_T)$. By homogeneity of $M_T$, we find a copy $Y''$ of $Y'$ with $v,w$ in place of $x,x'$ respectively, and a vertex $u$ in place of $y$. The image of $Y''$ under $\gamma$ is a oriented graph on $|Y|$ vertices with $v\gamma\to u\gamma$, that preserves all arcs involving $v$ and $u$. It follows that $Y''\gamma \cong Y$ but this is a contradiction as $Y$ does not belong to the age of $M_T$. So $\gamma$ must be injective.

Now assume that there exists an independent pair of vertices $v,w\in M_T$ such that $v\gamma\rarr{}w\gamma$. Select a $Y\in T$ in the same fashion as before, choose two vertices $x\rarr{}y$ of $Y$ and remove this arc to obtain a oriented graph $Z$ that embeds in $M_T$. Via homogeneity of $M_T$, we find an isomorphic copy of $Z$ with $v,w$ in place of $x,y$ respectively. Hence the image of $Z$ under $\gamma$ induces a copy of $Y$ in $M_T$; a contradiction as $Y$ does not belong to the age of $M_T$. So $M_T$ is a core for any such set of tournaments $T$ and therefore $M_T$ is not MB-homogeneous by \autoref{123lem}.
\end{example}

\begin{example} Let $S(3)$ be the \emph{myopic local order}, defined as follows. Distribute $\aleph_0$ many points densely around the unit circle such that for every point $a$ there are no points $b,c$ such that $\ang{a}{b} = \ang{b}{c} = \ang{c}{a} = \frac{2\pi}{3}$. Draw an arc $a\rarr{}b$ if and only if $\ang{a}{b} < 2\pi/3$; note that this means that $S(3)$ embeds no directed 3-cycles. We show that this structure is a core.

Assume that there is an endomorphism $\gamma$ of $S(3)$ with $a\gamma\rarr{}b\gamma$ for some independent pair of points $a,b\in S(3)$. As this occurs, we have that both $\ang{a}{b}$ and $\ang{b}{a} >2\pi/3$. As $\ang{a}{b} = 2\pi - \ang{b}{a}$, it follows that $\ang{b}{a}<4\pi/3$. From this, there exists a point $c$ such that $\ang{b}{c} = \ang{c}{a} <2\pi/3$ and so $b\rarr{}c$ and $c\rarr{}a$. The endomorphism $\gamma$ then creates a directed 3-cycle (or a loop) and this is a contradiction. Now suppose that for some non-related pair $a,b\in S(3)$, there is an endomorphism $\gamma$ such that $a\gamma = b\gamma$. As before, we can find a point $c\in S(3)$ such that $b\rarr{}c$ and $c\rarr{}a$. As $\gamma$ is an endomorphism it must preserve these relations and so $b\gamma\rarr{}c\gamma$ and $c\gamma\rarr{}a\gamma = b\gamma$. Hence we have a directed 2-cycle and this is obviously false; so $S(3)$ is a core. Applying \autoref{123lem} again implies that $S(3)$ is not MB-homogeneous.
\end{example}

Finally, we note that as \autoref{mepbep} is an if and only if statement, it is enough to prove that a structure $\mc{M}$ does not have the MEP or the AMEP to show that $\mc{M}$ is not MB-homogeneous.

\begin{example}\label{infree}
Let $D_n$ be the generic $I_n$-free digraph for $n\geq 3$; this is a homogeneous digraph \cite{cherlin1998classification}. We show here that $D_n$ does not have the AMEP, and hence is not MB-homogeneous by \autoref{mepbep}
So let $T$ be some tournament on $n-1$ vertices, and let $U$ be the disjoint union of $T$ with a vertex $u$. We note that $T\subseteq U\in$ Age$(D_n)$. Let $\bar{f}:T\rarr{}D_n$ be an antimonomorphism sending $T$ to an independent set of $n-1$ vertices in $D_n$; such a substructure exists by construction. Then as antimonomorphisms preserve non-arcs, a potential image point for $u$ in $D_n$ must be independent of $T\bar{f}$; this cannot happen as then $D_n$ would induce an independent $n$-set. So $D_n$ does not have the AMEP and hence is not MB-homogeneous by \autoref{mepbep}.
\end{example}

\begin{remark}
The argument in this example can be applied to show that any countable MB-homogeneous digraph is either a tournament (see \autoref{tournament}) or contains an infinite independent set $I_{\aleph_0}$.
\end{remark}

\begin{example}\label{knfreecomplement} Let $H$ be the complement of the homogeneous $K_n$-free graph for $n\geq 3$. We note that $H$ contains $R$ as a spanning subgraph (as it is algebraically closed) and so it is MM-homogeneous as it satisfies the MEP \cite{cameron2006homomorphism}. However, as $H$ does not embed an independent $n$-set, it does not contain an infinite null graph. As $H$ is not complete (as it has a non-edge), it follows that $H$ is not MB-homogeneous by \autoref{infcompnull}.
\end{example}

Following this final example, we are able to present a complete classification of those homogeneous graphs that are also MB-homogeneous.

\begin{theorem}\label{classification} The only countably infinite homogeneous and MB-homogeneous graphs are the following:\begin{itemize}
\item $K^{\aleph_0}$ and its complement $\overline{K}^{\aleph_0}$;
\item $\bigsqcup_{i\in\mb{N}}K_i^{\aleph_0}$ and its complement $\overline{\bigsqcup_{i\in\mb{N}}K_i^{\aleph_0}}$;
\item the random graph $R$.
\end{itemize}
\end{theorem}

\begin{proof} We check every item in the classification of countably infinite, undirected homogeneous graphs given in \cite{lachlan1980countable}. We showed that those graphs on the list are MB-homogeneous in \autoref{kcgraph}, \autoref{dunionkgr} and \autoref{random}. Any other disconnected homogeneous graph must be the a countable union of complete graphs or a finite union of infinite complete graphs; these are not MB-homogeneous by \autoref{dunionkgr}. The only other countable homogeneous graphs are the $K_n$-free graphs and their complements; these are not MB-homogeneous by \autoref{knfreegraphs} and \autoref{knfreecomplement}.
\end{proof}

We end on some open questions concerning  MB-homogeneous graphs in general. As we have seen, it would be an arduous task to classify MB-homogeneous graphs up to isomorphism; the idea of equivalence up to bimorphisms (see \autoref{biequiv}) represents the best hope to identify all MB-homogeneous graphs. A positive answer to the following question would constitute the best classification result possible for MB-homogeneous graphs, given the amount and range of examples above.

\begin{question} Is every countable MB-homogeneous graph bimorphism equivalent to one of the five graphs in \autoref{classification}?
\end{question}

A weaker version of this question is the following:

\begin{question} Are there only countably many MB-homogeneous graphs up to bimorphism equivalence?
\end{question}

Finally, we notice that every example of an MB-homogeneous graph in this article is also HE-homogeneous. This motivates our final question:

\begin{question} Is there a countably infinite MB-homogeneous graph that is not HE-homogeneous? Conversely, is there a HE-homogeneous graph that is not MB-homogeneous?
\end{question}

\begin{acknowledgements}
The authors would like to thank the referees for their 
valuable comments which have helped to improve the article.  
The research of R.\ D. Gray was supported by the EPSRC grant EP/N033353/1 ``Special inverse monoids: subgroups, structure, geometry, rewriting systems and the word problem''.
\end{acknowledgements}

\bibliographystyle{abbrv}

%
%
%\bibliography{masterbib.bib}
\end{document}